\documentclass{article}
\usepackage{amsmath,amsthm,amssymb,bbm}

\newcommand{\assign}{:=}
\newcommand{\tmop}[1]{\ensuremath{\operatorname{#1}}}
\newcommand{\tmstrong}[1]{\textbf{#1}}

\newtheorem{theorem}{Theorem}
\newtheorem{lemma}{Lemma}

\newcommand{\XXint}[3]{{\setbox}0=\text{\ensuremath{#1 #2 #3 \int}}
{\vcenter{\text{\ensuremath{#2 #3}}}}{\kern}-.5{\tmwd}0}

\newcommand{\opn}[2]{\newcommand{\1}{\}} {\opn}{\Rm{Rm}} {\opn}{\Ric{Ric}}
{\opn}{\Rc{Rc}} {\opn}{\Scal{Sc}} {\opn}{\Tr{Tr}} {\opn}{\Trac{Tr}}
{\opn}detdet {\opn}{\diam{diam}} {\opn}{\dist{dist}} {\opn}{\Im}Im
{\opn}{\div}div {\opn}{\Ker{Ker}} {\opn}expexp {\opn}{\Vol{Vol}}
{\opn}{\exph{exph}} {\opn}{\Herm{Herm}} {\opn}{\End{End}} {\opn}{\Hess{Hess}}
{\opn}{\Vol{Vol}}}

\newcommand{\R}{\mathbb{R}}
\newcommand{\C}{\mathbb{C}}
\newcommand{\K}{\mathbb{K}}
\newcommand{\I}{\mathbb{I}}

\newcommand{\contract}{{\kern}-1.5pt{\vrule} width6.0pt height0.4pt depth0pt
{\vrule} width0.4pt height4.0pt depth0pt}
\newcommand{\retract}{{\kern}-1.5pt{\vrule} width0.4pt height4.0pt depth0pt
{\vrule} width6.0pt height0.4pt depth0pt}
\newcommand{\Openbox}{{\leavevmode} {\text{{\hfil}{\vrule}
width{\boxrulethickness} {\vbox} to{\Openboxwidth{{\advance}{\Openboxwidth}
-2{\boxrulethickness} {\hrule} height {\boxrulethickness}
width{\Openboxwidth}{\vfil} {\hrule} height{\boxrulethickness}}}{\vrule}
width{\boxrulethickness}{\hfil} }}}

\begin{document}

\title{ A second variation formula for Perelman's $\mathcal{W}$-functional along the modified K\"ahler-Ricci flow
}
\author{\\
{\tmstrong{NEFTON PALI}}}\maketitle

\begin{abstract}
  We show a quite simple second variation formula for Perelman's
  $\mathcal{W}$-functional along the modified K\"ahler-Ricci flow over Fano
  manifolds.
\end{abstract}

\section{Introduction}

In a celebrated preprint Grigory Perelman \cite{Per} introduced an entropy type functional
denoted by $\mathcal{W}$ which gradient flow coincides with Perelman's modified
Ricci flow. In the Fano case this flow coincides with Perelman's modified
K\"ahler-Ricci flow. The total second variation of this functional plays an
important role in the study of the stability and convergence of the
K\"ahler-Ricci flow over Fano manifolds. Total second variation formulas at a 
shrinking Ricci soliton point were obtained independently by 
Cao-Hamilton-Ilmanenen \cite{C-H-I}, Cao-Zhu \cite{Ca-Zhu} and Tian-Zhu \cite{Ti-Zhu}. 
The work of Cao-Zhu \cite{Ca-Zhu} is based on the previous work of Cao-Hamilton-Ilmanenen \cite{C-H-I}.
Important applications to the stability and the convergence 
of the K\"ahler-Ricci flow over Fano manifolds were 
given by Tian-Zhu \cite{Ti-Zhu} and by Tian-Zhang-Zhang-Zhu \cite{T-Z-Z-Z}.
Total second variation formulas at arbitrary points in the space of metrics can be found in \cite{Pal}. 
The formulas in \cite{Pal} are of different nature from the ones obtained by \cite{C-H-I}, \cite{Ca-Zhu} and \cite{Ti-Zhu}.

In this paper we show a very simple second
variation formula for Perelman's $\mathcal{W}$-functional along the modified
K\"ahler-Ricci flow over Fano manifolds. 

Our computation concern the
K\"ahler-Ricci flow but we feel that the second variation formula is more
meaningful for the modified K\"ahler-Ricci flow. In any case one can switch
easily from one formula relative to a flow to the other thanks to the invariance by
diffeomorphisms of the functional $\mathcal{W}$. It is some how surprising that we could not
prove a simple second variation formula by working directly with the modified
K\"ahler-Ricci flow. 

The key computation which allows us to obtain our second variation formula
is based on a very surprising integration by parts formula along the K\"ahler-Ricci flow.
It turns out that the corresponding version of this formula along the modified K\"ahler-Ricci 
flow can be obtained using Perelman's version of the twice contracted Bianchi identity. 
The details about this last fact are explained in section \ref{modKRF-Magic-Heat}.

We expect convexity of the functional $\mathcal{W}$ along the K\"ahler-Ricci flow.
Our formula should be considered as a first step in this direction.
It should be clear for the experts that convexity along the flow implies quite natural and strong convergence results.
We explain in detail below the set-up and our
result.

In this paper we will adopt the sign convention $\Delta \assign \tmop{div}
\nabla$ in order to not generate confusion in some standard formulas along the
K\"ahler-Ricci flow.

Let $(X, g)$ be a compact oriented (for simplicity) Riemannian manifold of
dimension $2 n$ and $f$ a smooth real valued function over $X$. We remind that
Perelman's $\mathcal{W}$-functional \cite{Per} is defined (up to a constant) by the formula
\begin{eqnarray*}
  \mathcal{W} (g, f) & : = & \int_X \Big( | \nabla_g f|^2_g \hspace{0.75em} +
  \hspace{0.75em} \tmop{Scal}_g \hspace{0.75em} + \hspace{0.75em} 2 \,f
  \hspace{0.75em} - \hspace{0.75em} 2 \,n \Big) e^{- f} dV_g .
\end{eqnarray*}
It rewrites also as
\begin{eqnarray*}
  \mathcal{W} (g, f) & = & \int_X \Big( \Delta_g \,f \hspace{0.75em} +
  \hspace{0.75em} \tmop{Scal}_g \hspace{0.75em} + \hspace{0.75em} 2 \,f
  \hspace{0.75em} - \hspace{0.75em} 2 \,n \Big) e^{- f} dV_g\\
  &  & \\
  & = & \int_X 2\; H (g, f) \hspace{0.25em} e^{- f} dV_g \hspace{0.25em} .
\end{eqnarray*}
where
\begin{eqnarray*}
  2 \;H (g, f) & \assign & 2 \;\Delta_g \,f \hspace{0.75em} - \hspace{0.75em} |
  \nabla_g f|_g^2 \hspace{0.75em} + \hspace{0.75em} \tmop{Scal}_g
  \hspace{0.75em} + \hspace{0.75em} 2 \,f \hspace{0.75em} - \hspace{0.75em} 2 n
  \hspace{0.25em} .
\end{eqnarray*}
(We use here the standard identity $\Delta_g e^{- f} = (| \nabla_g f|^2_g -
\Delta_g f) e^{- f}$ .) The importance of the therm $H$ is known from
Perelman's work \cite{Per} (see the subsection \ref{Per-first-varKRF} in the appendix) and it will appear also in our second variation
formula.

Let $(X, J_0)$ be a Fano manifold and let $( \hat{g}_t)_{t \geqslant 0}$ be
the K\"ahler-Ricci flow
\begin{eqnarray*}
  \frac{d}{dt}  \;\hat{g}_t & = & - \;\;\tmop{Ric} ( \hat{g}_t) \hspace{0.75em} +
  \hspace{0.75em} \hat{g}_t,
\end{eqnarray*}
with associated symplectic form $\hat{\omega}_t : = \hat{g}_t J_0 \in 2 \pi
c_1 (X)$. Consider also a solution $( \hat{f}_t)_{t \in [0, T]}$ of Perelman's
backward heat equation
\begin{equation}
  \label{RET-gaug-pot-EVOL} 2 \frac{d}{dt} \hspace{0.25em}  \hat{f}_t
  \hspace{0.75em} = \hspace{0.75em} - \;\;\Delta_{\hat{g}_t}  \hat{f}_t
  \hspace{0.75em} + \hspace{0.75em} | \nabla_{\hat{g}_t}  \hat{f}_t
  |_{\hat{g}_t}^2 \hspace{0.75em} - \hspace{0.75em} \tmop{Scal}_{\hat{g}_t}
  \hspace{0.75em} + \hspace{0.75em} 2 n,
\end{equation}
and set $\Omega \assign e^{- \hat{f}_0} d V_{\hat{g}_0}$. Consider now the
flow of diffeomorphisms $(\Psi_t)_{t \in \left[ 0, T] \right.}$ given by the
equation
\begin{equation}
  \label{diff-flw} 2 \,\frac{d}{dt} \; \Psi_t 
\;\; = \;\;
-\;\;\left( \nabla_{\hat{g}_t}  \hat{f}_t \right) \circ \Psi_t,
\end{equation}
and set $g_t \assign \Psi_t^{\ast}  \hat{g}_t, J_t \assign \Psi_t^{\ast} J_0,
\omega_t \assign \Psi_t^{\ast}  \hat{\omega}_t, f_t \assign \hat{f}_t \circ
\Psi_t$. Then hold the evolution formulas
\begin{eqnarray*}
  \dot{g}_t \;\;\assign\;\; \frac{d}{dt} g_t & = & -\;\; \tmop{Ric} (g_t) \;\;- \;\;\nabla_{g_t} d\,
  f_t \;\;+\;\; g_t\;,\\
  &  & \\
  \dot{\omega}_t \;\;\assign \;\;\frac{d}{dt} \omega_t & = & - \;\;\tmop{Ric}_{_{J_t}}
  (\omega_t) \;-\; i \hspace{0.25em} \partial_{_{J_t}}
  \overline{\partial}_{_{J_t}} f_t \;+\; \omega_t\;,
\end{eqnarray*}
and
\begin{eqnarray*}
  2 \frac{d}{dt} \; f_t & = & -\; \Delta_{g_t} f_t \;
  - \; \tmop{Scal} (g_t) \; + \; 2 n\; .
\end{eqnarray*}
This last rewrites as
\begin{eqnarray*}
  2 \frac{d}{dt} \hspace{0.25em} f_t & = & \tmop{Tr}_{\omega_t}  \frac{d}{dt}
  \hspace{0.25em} \omega_t\;,
\end{eqnarray*}
which is equivalent to the volume form preserving condition $e^{- f_t} d
V_{g_t} = \Omega$.
The flow of K\"ahler structures $(X, J_t, \omega_t)_{t \in [0, T]}$ is called
Perelman's modified K\"ahler-Ricci flow. We set $\mathcal{W}_t : = \mathcal{W}
(g_t, f_t)$ along the modified K\"ahler-Ricci flow and we remind Perelman's \cite{Per}
fundamental identity
\begin{equation}
  \label{Per-var}  \dot{\mathcal{W}}_t \;\; =\;\; \int_X \left|
  \dot{g}_t \right|^2_t \Omega \;.
\end{equation}
Indeed this is equivalent to Perelman's monotony \cite{Per}
of the
$\mathcal{W}$-functional along the K\"ahler-Ricci flow thanks to the
invariance of the $\mathcal{W}$-functional under the action of
diffeomorphisms. In our case it gives $\mathcal{W}_t = \mathcal{W} ( \hat{g}_t, \hat{f}_t)$.
With this notations our result states as follows.
\begin{theorem}
  \label{W-secVar}Along the modified K\"ahler-Ricci flow $(X, J_t,
  \omega_t)_{t \in [0, T]}$ hold the second variation formula
  \begin{eqnarray*}
    \ddot{\mathcal{W}}_t & = & \int_X \left[ 2 \,\big\langle \dot{\omega}_t
    \cdot \tmop{Rm}_{g_t}, \dot{\omega}_t \big\rangle_t \;\,+\;\, 2 \,\big| \nabla_t
    \,H_t \big|^2_t \;\,-\;\, \big| \nabla_t  \,\dot{g}_t \big|^2_t \right] \Omega\;,
  \end{eqnarray*}
  where $\tmop{Rm}_{g_t}\in C^{\infty} (X, \tmop{End} (\Lambda^2 T_X))$ denotes
  the Riemann curvature operator and $H_t \assign H (g_t, f_t) .$
\end{theorem}

\section{The second variation of $\mathcal{W}$ along the \\K\"ahler-Ricci flow}

All the computations in this section are done with respect to the
K\"ahler-Ricci flow $(X, J_0, \hat{g}_t)$. From now on all the complex operators
will depend on the complex structure $J_0$. 
Therefore we will suppress the
dependence on $J_0$ when it is redundant. Moreover all the covariant
derivatives and norms are computed with respect to the K\"ahler-Ricci flow.
Thus for notation simplicity we will drop the dependence on the evolving
metric and on the time. Along the K\"ahler-Ricci flow we define the heat
operator 
$$
\Box \;\;: =\;\; \Delta \;\;-\;\; 2\; \frac{\partial}{\partial t}
$$ 
and the
conjugate heat operator
\[ \Box^{\ast} \hspace{0.75em} : = \hspace{0.75em} \Delta \hspace{0.75em} +
   \hspace{0.75em} 2 \hspace{0.25em} \frac{\partial}{\partial t}
   \hspace{0.75em} - \hspace{0.75em} \tmop{Scal} \hspace{0.75em}\;\; +\;\;
   \hspace{0.75em} 2 n \hspace{0.25em} . \]
The terminology is justified by the formula
\[ 2 \hspace{0.25em}  \frac{d}{dt} \int_X (a \hspace{0.25em} b)
   \hspace{0.25em} d V_{\hat{g}} \;\;=\;\; - \;\;\int_X \left( \Box \hspace{0.25em} a
   \hspace{0.25em} b \;\;-\;\; a \hspace{0.25em} \Box^{\ast} b \right) d V_{\hat{g}} 
   \hspace{0.25em}, \]
for any $a, b \in C^{\infty} (X \times \R_{\ge 0} \hspace{0.25em}, \R)$. We
observe that with this notations the backward heat equation
(\ref{RET-gaug-pot-EVOL}) is equivalent to the equation 
\begin{eqnarray}
\label{BOX-pot-EVOL}
\Box^{\ast} \,e^{-
\hat{f}} \;\;= \;\; 0\;.
\end{eqnarray}
For any $\alpha \in \Lambda_{_{\mathbbm{R}}}^2 T^{\ast}_X$ we define the
endomorphism $\alpha^{\ast} \assign \hat{\omega}^{- 1} \alpha$. Let now
for notation simplicity
\begin{eqnarray*}
  \tmop{Ric} & \assign & \tmop{Ric}_{_{J_0}} ( \hat{\omega}_t)\;,\\
  &  & \\
  \underline{\tmop{Ric}} & \assign & \tmop{Ric} \;\;+\;\; i \partial
  \overline{\partial}  \hat{f}\;,\\
  &  & \\
  \alpha & \assign & \hat{\omega} \;\; - \;\;
  \underline{\tmop{Ric}}\;,\\
  &  & \\
  A & \assign & \overline{\partial} \,\nabla \hat{f} \;.
\end{eqnarray*}
Thus time deriving Perelman's monotony formula for the $\mathcal{W}$-functional
along the K\"ahler-Ricci flow
\begin{eqnarray*}
  \dot{\mathcal{W}} & = & \int_X \left( | \alpha^{\ast} |^2 \;\;+\;\; |A|^2 \right)
  e^{- \hat{f}} d V_{\hat{g}}\;,
\end{eqnarray*}
(see formula (\ref{W-frst}) in the subsection \ref{Per-first-varKRF} of the appendix) we obtain
\begin{equation}
  \label{pre-SEcW}  \ddot{\mathcal{W}}\;\; =\;\; -\;\; \hspace{0.25em} \frac{1}{2}
  \hspace{0.75em} \int_X \Box \left( | \alpha^{\ast} |^2 \;\;+\;\; |A|^2 \right) e^{-
  \hat{f}} d V_{\hat{g}}  \hspace{0.25em} .
\end{equation}
In the next subsections we will expand the integrand therm.

\subsection{Computation of the time variation of the squared\\ norms}

We introduce first a notation. Let $E : = \tmop{End}_{_{\C}} (T_{_{X,
J_0}})$. For all $\xi \in T_X$ we define the $J_0$-anti-linear operator
$\overline{\partial}_{_E} \tmop{Ric}^{\ast} \cdot \hspace{0.25em} \xi$ as
\[ \left( \overline{\partial}_{_E} \tmop{Ric}^{\ast} \cdot \hspace{0.25em} \xi
   \right) (\eta) : = \left( \overline{\partial}_{_{E, \eta}}
   \tmop{Ric}^{\ast} \right) (\xi) \hspace{0.25em}, \hspace{1em} \forall \eta
   \in T_X \hspace{0.25em} . \]
The evolution equation of the Ricci form $\tmop{Ric}$ and the evolution
equation (\ref{RET-gaug-pot-EVOL}) of $\hat{f}$ imply the identity
\[ 2 \hspace{0.25em} \frac{d}{dt} \; \underline{\tmop{Ric}}
   \hspace{0.75em} = \hspace{0.75em} i \partial \bar{\partial} \left(
   \tmop{Scal} \hspace{0.75em} + \hspace{0.75em} 2 \hspace{0.25em} 
   \frac{d}{dt} \hspace{0.25em} \hat{f} \right) \hspace{0.75em} =
   \hspace{0.75em} i \partial \overline{\partial} \left( | \nabla \hat{f} |^2
   \hspace{0.75em} - \hspace{0.75em} \Delta \hat{f} \right) \hspace{0.25em} .
\]
Moreover
\begin{eqnarray*}
  \frac{d}{dt}  \;	 \underline{\tmop{Ric}} & = & \left(
  \frac{d}{dt} \; \,\hat{\omega} \right)  \hspace{0.25em}
  \underline{\tmop{Ric}}^{\ast}_{} \hspace{0.75em} + \hspace{0.75em}
  \hat{\omega} \; \frac{d}{dt}  \;
  \underline{\tmop{Ric}}^{\ast}_{}\\
  &  & \\
  & = & \underline{\tmop{Ric}} \hspace{0.75em} - \hspace{0.75em} \tmop{Ric}
  \hspace{0.25em} \underline{\tmop{Ric}}^{\ast} \hspace{0.75em} +
  \hspace{0.75em}  \hat{\omega} \; \frac{d}{dt} \;
  \underline{\tmop{Ric}}^{\ast}_{}  \hspace{0.25em},
\end{eqnarray*}
and thus
\[ \frac{d}{dt}  \; \underline{\tmop{Ric}}^{\ast} \hspace{0.75em}
   = \hspace{0.75em}  \left( \frac{d}{dt}  \;
   \underline{\tmop{Ric}}_{} \right)^{\ast} \hspace{0.75em} + \hspace{0.75em}
   \tmop{Ric}^{\ast}  \hspace{0.25em}  \underline{\tmop{Ric}}^{\ast}
   \hspace{0.75em} - \hspace{0.75em}  \underline{\tmop{Ric}}^{\ast} 
   \hspace{0.25em} . \]
The fact that the endomorphism $\alpha^{\ast} = \I \hspace{0.25em} -
\hspace{0.25em} \underline{\tmop{Ric}}^{\ast}$ is $\hat{g}$-symmetric implies the identity
$| \alpha^{\ast} |^2 = \tmop{Tr}_{_{\R}} (\alpha^{\ast})^2$. Time deriving
this identity we infer
\begin{eqnarray*}
  \frac{d}{dt} \;| \alpha^{\ast} |^2 & = & - \hspace{0.25em}
  \tmop{Tr}_{_{\R}} \left( \frac{d}{dt}  \;
  \underline{\tmop{Ric}}^{\ast} \alpha^{\ast} \hspace{0.75em} +
  \hspace{0.75em} \alpha^{\ast}  \hspace{0.25em} \frac{d}{dt}  \;
  \underline{\tmop{Ric}}^{\ast}_{} \right)\\
  &  & \\
  & = & - \;\;2 \tmop{Tr}_{_{\R}} \left( \alpha^{\ast}  \hspace{0.25em}
  \frac{d}{dt}  \; \underline{\tmop{Ric}}^{\ast}_{} \right)\\
  &  & \\
  & = & - \;\;2 \left\langle \alpha^{\ast}, \left( \frac{d}{dt}  \;
  \underline{\tmop{Ric}}_{} \right)^{\ast} \hspace{0.75em} + \hspace{0.75em}
  \tmop{Ric}^{\ast}  \hspace{0.25em} \underline{\tmop{Ric}}^{\ast}_{}
  \hspace{0.75em} - \hspace{0.75em}  \underline{\tmop{Ric}}^{\ast}_{}
  \right\rangle \hspace{0.25em} .
\end{eqnarray*}
In order to obtain the third equality we need to observe the elementary
identities
\begin{eqnarray}\label{Sym-Tr-Rc}
\left\langle \tmop{Ric}^{\ast}  \;
  \underline{\tmop{Ric}}^{\ast}_{}, \alpha^{\ast}  \right\rangle 
&=&
  \tmop{Tr}_{_{\R}}  \left( \tmop{Ric}^{\ast} 
  \underline{\tmop{Ric}}^{\ast} \alpha^{\ast}  \right) \nonumber
\\\nonumber
\\
&=& \tmop{Tr}_{_{\R}} 
  \left( \alpha^{\ast} \tmop{Ric}^{\ast}  \hspace{0.25em}
  \underline{\tmop{Ric}}^{\ast}_{} \right) \nonumber
\\\nonumber
\\
&=& \left\langle \alpha^{\ast},
  \underline{\tmop{Ric}}^{\ast}_{} \tmop{Ric}^{\ast}  \right\rangle 
  \hspace{0.25em} .
\end{eqnarray} 
In conclusion hold the identity
\begin{equation}
  \label{der-Nr-Rc}  \frac{d}{dt} \; | \alpha^{\ast} |^2 
\;\;=\;\;
  \left\langle \alpha^{\ast}, \left[ i \partial \overline{\partial} (\Delta
  \hat{f} \;\, - \;\, | \nabla \hat{f} |^2)
  \right]^{\ast} \;\, - \;\, 2 \hspace{0.25em}
  \tmop{Ric}^{\ast}  \hspace{0.25em} \underline{\tmop{Ric}}^{\ast}
  \;\, + \;\, 2 \hspace{0.25em}  \underline{\tmop{Ric}}
  \hspace{0.25em} \right\rangle  \hspace{0.25em} .
\end{equation}
We compute now the time variation of the operator $A$. Time deriving the
definition of the gradient of $\hat{f}$ we infer
\begin{eqnarray*}
  d \left( \frac{d}{dt} \hspace{0.25em}  \hat{f} \right) & = & \left(
  \frac{d}{dt} \hspace{0.25em} \nabla \hat{f} \right) \neg \;\hat{g}
  \hspace{0.75em} + \hspace{0.75em} \nabla \hat{f}_t \;\neg\; \frac{d}{dt} 
  \hspace{0.25em} \hat{g}\\
  &  & \\
  & = & \left( \frac{d}{dt} \hspace{0.25em} \nabla \hat{f} \right) \neg\;
  \hat{g} \hspace{0.75em} + \hspace{0.75em} d \hat{f} \hspace{0.75em} -
  \hspace{0.75em} \nabla \hat{f} \;\neg\; \tmop{Ric} ( \hat{g}) \hspace{0.25em},
\end{eqnarray*}
and thus using the elementary identity $\tmop{Ric} ( \hat{g}) = -
\hspace{0.25em} \tmop{Ric} \cdot J_0$ we obtain
\begin{eqnarray*}
  \frac{d}{dt} \hspace{0.25em} \nabla \hat{f} & = & \nabla \left( \frac{d}{dt}
  \hspace{0.25em}  \hat{f} \hspace{0.75em} - \hspace{0.75em} \hat{f} \right)
  \hspace{0.75em} + \hspace{0.75em} \tmop{Ric}^{\ast} \cdot \nabla \hat{f}
  \hspace{0.25em} .
\end{eqnarray*}
We deduce the expression
\[ \dot{A} \;\;: =\;\; \frac{d}{dt} \hspace{0.25em} A \hspace{0.75em} =
   \hspace{0.75em} \overline{\partial} \hspace{0.25em} \nabla \left(
   \frac{d}{dt} \hspace{0.25em}  \hat{f} \hspace{0.75em} - \hspace{0.75em}
   \hat{f} \right) \hspace{0.75em} + \hspace{0.75em} \overline{\partial}_{_E}
   \tmop{Ric}^{\ast} \cdot \nabla \hat{f} \hspace{0.75em} + \hspace{0.75em}
   \tmop{Ric}^{\ast} A \hspace{0.25em} . \]
On the other hand $|A|^2 = \tmop{Tr}_{_{\R}} A^2$ since $A$ is also
$\hat{g}$-symmetric. Time deriving this identity we infer the equalities
\begin{eqnarray*}
  \frac{d}{dt} \hspace{0.75em} |A|^2 & = & \tmop{Tr}_{_{\R}}  \left( \dot{A} A
  \hspace{0.75em} + \hspace{0.75em} A \dot{A} \right)\\
  &  & \\
  & = & 2 \left\langle A, \overline{\partial} \hspace{0.25em} \nabla \left(
  \frac{d}{dt} \hspace{0.25em}  \hat{f} \hspace{0.75em} - \hspace{0.75em}
  \hat{f} \right) \hspace{0.75em} + \hspace{0.75em} \overline{\partial}_{_E}
  \tmop{Ric}^{\ast} \cdot \nabla \hat{f} \right\rangle\\
  &  & \\
  & + & 2 \left\langle A^2, \tmop{Ric}^{\ast} \right\rangle \hspace{0.25em} .
\end{eqnarray*}
In order to obtain the last therm we use the identity
\begin{equation}
  \label{Sym-Tr-A} \tmop{Tr}_{_{\R}} \left( A \tmop{Ric}^{\ast} A \right) =
  \left\langle A^2, \tmop{Ric}^{\ast} \right\rangle  \hspace{0.25em} .
\end{equation}
We infer from the evolution equation (\ref{RET-gaug-pot-EVOL}) of $\hat{f}$
the formula
\begin{eqnarray}
  \label{Ev-Nrm-A}  \frac{d}{dt} \hspace{0.75em} |A_{} |^2 & = & \left\langle
  A, \overline{\partial} \hspace{0.25em} \nabla \left( | \nabla \hat{f} |^2
  \hspace{0.75em} - \hspace{0.75em} \Delta \hspace{0.25em} \hat{f}
  \hspace{0.75em} - \hspace{0.75em} \tmop{Scal} \hspace{0.75em} -
  \hspace{0.75em} 2 \hspace{0.25em} \hat{f} \right) \right\rangle \nonumber\\
  &  &  \nonumber\\
  & + & 2 \hspace{0.25em}  \left\langle A, \overline{\partial}_{_E}
  \tmop{Ric}^{\ast} \cdot \nabla \hat{f} \right\rangle \hspace{0.75em} +
  \hspace{0.75em} 2 \left\langle A^2, \tmop{Ric}^{\ast} \right\rangle 
  \hspace{0.25em} . 
\end{eqnarray}

\subsection{Computation of the Laplacian of the squared norms}

At an arbitrary space time point $(x_0, t_0) \in X \times [0, T]$ we fix
$J_0$-holomorphic and $\hat{g}_{t_0}$-geodesic coordinates $(z_1, ..., z_n)$
centered at the point $x_0$ and we set $\zeta_k : = \frac{\partial}{\partial
z_k}$. At the time $t_0$ let $\nabla_{r} : = \nabla_{\zeta_r}$, $\nabla_{\bar{r}} : =
\nabla_{\bar{\zeta}_r}$ and set $e_r : = \zeta_r \; +
\; \bar{\zeta}_r$. Expanding and canceling we obtain the
expression
\begin{equation}
  \label{bas-Lap} \Delta \hspace{0.75em} = \hspace{0.75em} \nabla_{e_r}
  \nabla_{e_r} \hspace{0.75em} + \hspace{0.75em} \nabla_{J_0 e_r}
  \nabla_{J_0 e_r} \hspace{0.75em} = \hspace{0.75em} 2 \hspace{0.25em}
  \nabla_{r} \hspace{0.25em} \nabla_{\bar{r}} \hspace{0.75em} +
  \hspace{0.75em} 2 \hspace{0.25em} \nabla_{\bar{r}} \hspace{0.25em}
  \nabla_{r} \;.
\end{equation}
{\bf{Part I.}} Let
\begin{eqnarray*}
  R_{k, \bar{l}} & \assign & \tmop{Ric} (\zeta_k, \bar{\zeta}_l)\;,\\
  &  & \\
  R'_{k, \bar{l}} & \assign & \underline{\tmop{Ric}} (\zeta_k, \bar{\zeta}_l)\;
  .
\end{eqnarray*}
With this notations hold the local expression
\[ \underline{\tmop{Ric}}^{\ast} \hspace{0.75em} = \hspace{0.75em}
   \mathcal{R}_{k, \bar{l}} \;\, \zeta_k^{\ast} \otimes \zeta_l \;\;+\;\;
   \overline{\mathcal{R}_{k, \bar{l}}}  \;\,\bar{\zeta}_k^{\ast} \otimes
   \bar{\zeta}_l\;, 
\hspace{2em} \mathcal{R}_{k, \bar{l}} \hspace{0.75em} =
   \hspace{0.75em} 2 \hspace{0.25em} R'_{k, \bar{p}} \hspace{0.25em} 
   \hat{\omega}^{p, \bar{l}}\;, \]
where 
$$
\hat{\omega} \;\;=\;\; \frac{i}{2} \;\omega_{k, \bar{l}} \;\zeta_k^{\ast}
\wedge \bar{\zeta}^{\ast}_l\;.
$$ 
Using the expression (\ref{bas-Lap}), the
vanishing properties of the complexified Levi-Civita connection and the
previous identities we infer the expressions at the space time point $(x_0,
t_0)$
\begin{eqnarray*}
  \Delta \hspace{0.25em}  \underline{\tmop{Ric}}^{\ast} & = & 4
  \hspace{0.25em} \partial^2_{r, \bar{r}} \hspace{0.25em} \mathcal{R}_{k,
  \bar{l}} \;\, \zeta_k^{\ast} \otimes \zeta_l \hspace{0.75em} +
  \hspace{0.75em} 4 \hspace{0.25em} \partial^2_{r, \bar{r}} \;
  \overline{\mathcal{R}_{k, \bar{l}}}  \;\,  \bar{\zeta}_k^{\ast}
  \otimes \bar{\zeta}_l\\
  &  & \\
  & + & 2 \hspace{0.25em}  \mathcal{R}_{k, \bar{l}} \hspace{0.25em}
  \nabla_{\bar{r}} \hspace{0.25em} \nabla_{r} \hspace{0.25em}
  \zeta^{\ast}_k \otimes \zeta_l \hspace{0.75em} + \hspace{0.75em} 2
  \hspace{0.25em}  \mathcal{R}_{k, \bar{l}} \hspace{0.25em} \zeta_k^{\ast}
  \otimes \nabla_{\bar{r}} \hspace{0.25em} \nabla_{r} \hspace{0.25em}
  \zeta_l\\
  &  & \\
  & + & 2 \hspace{0.25em}  \mathcal{R}_{l, \bar{k}} \hspace{0.25em}
  \nabla_{r} \hspace{0.25em} \nabla_{\bar{r}} \hspace{0.25em} 
  \bar{\zeta}^{\ast}_k \otimes \bar{\zeta}_l \hspace{0.75em} + \hspace{0.75em}
  2 \hspace{0.25em} \mathcal{R}_{l, \bar{k}} \hspace{0.25em} 
  \bar{\zeta}_k^{\ast} \otimes \nabla_{r} \hspace{0.25em} \nabla_{\bar{r}} \hspace{0.25em}  \bar{\zeta}_l\\
  &  & \\
  & = & 4 \hspace{0.25em} \partial^2_{r, \bar{r}} \hspace{0.25em}
  \mathcal{R}_{k, \bar{l}} \;\, \zeta_k^{\ast} \otimes \zeta_l
  \hspace{0.75em} + \hspace{0.75em} 4 \hspace{0.25em} \partial^2_{r, \bar{r}} \;
  \overline{\mathcal{R}_{k, \bar{l}}}  \;\,  \bar{\zeta}_k^{\ast}
  \otimes \bar{\zeta}_l\\
  &  & \\
  & + & 2 \left( R_{k, \bar{p}} \hspace{0.25em}  \mathcal{R}_{p, \bar{l}}
  \hspace{0.75em} - \hspace{0.75em}  \mathcal{R}_{k, \bar{p}} \hspace{0.25em}
  R_{p, \bar{l}} \right) \zeta_k^{\ast} \otimes \zeta_l\\
  &  & \\
  & + & 2 \left( \mathcal{R}_{l, \bar{p}} \hspace{0.25em} R_{p, \bar{k}}
  \hspace{0.75em} - \hspace{0.75em} R_{l, \bar{p}} \hspace{0.25em} 
  \mathcal{R}_{p, \bar{k}} \right)  \bar{\zeta}_k^{\ast} \otimes
  \bar{\zeta}_l\\
  &  & \\
  & = & 4 \hspace{0.25em} \partial^2_{r, \bar{r}} \hspace{0.25em}
  \mathcal{R}_{k, \bar{l}} \;\, \zeta_k^{\ast} \otimes \zeta_l
  \hspace{0.75em} + \hspace{0.75em} 4 \hspace{0.25em} \partial^2_{r, \bar{r}} \;
  \overline{\mathcal{R}_{k, \bar{l}}}  \;\,  \bar{\zeta}_k^{\ast}
  \otimes \bar{\zeta}_l\\
  &  & \\
  & + & \underline{\tmop{Ric}}^{\ast}_{} \tmop{Ric}^{\ast} \hspace{0.75em} -
  \hspace{0.75em} \tmop{Ric}^{\ast}  \underline{\tmop{Ric}}^{\ast}_{}
  \hspace{0.25em} .
\end{eqnarray*}
At this point we need the following elementary fact.
\begin{lemma}
  \label{Sim-Ric}At the space time point $(x_0, t_0)$ hold the identity
  \begin{eqnarray*}
    \partial^2_{r, \bar{r}} \hspace{0.25em}  \mathcal{R}_{k, \bar{l}} & = &
    \partial^2_{k, \bar{l}} \hspace{0.25em}  \mathcal{R}_{r, \bar{r}}
    \;\;-\;\;\mathcal{R}_{r, \bar{p}} \tmop{Rm}^{r, \bar{p}}_{k, \bar{l}} \;\;+\;\;
    2\,\mathcal{R}_{k, \bar{p}} R_{p, \bar{l}}\;,
  \end{eqnarray*}
  where
  \begin{eqnarray*}
    \tmop{Rm} & = & \tmop{Rm}^{r, \bar{p}}_{k, \bar{l}} \;(\zeta_k^{\ast} \wedge
    \bar{\zeta}^{\ast}_l) \otimes (\zeta_r \wedge \bar{\zeta}_p)\;,\qquad
    \tmop{Rm}^{r, \bar{p}}_{k, \bar{l}}\;\; =\;\; -\;\, 2 \,\partial^2_{k, \bar{l}}
    \hspace{0.25em}  \hat{\omega}_{p, \bar{r}}\; .
  \end{eqnarray*}
\end{lemma}
\begin{proof}
  Space deriving twice the expression of the coefficients $R_{k, \bar{l}}$ and
  using the K\"ahler symmetry relations we find at the space time point $(x_0,
  t_0)$ the expression
  \begin{eqnarray*}
    4 \hspace{0.25em} \partial^2_{j, \bar{k}}\; R_{h, \bar{s}} & = & -
    \;\;4\,  \partial^4_{j, \bar{k}, h, \bar{s}}
    \hspace{0.25em}  \hat{\omega}_{l, \bar{l}} \hspace{0.75em} +
    \hspace{0.75em} \tmop{Rm}_{l, \bar{r}}^{s, \bar{h}} \tmop{Rm}_{j,
    \bar{k}}^{l, \bar{r}} \hspace{0.75em} + \hspace{0.75em} \tmop{Rm}_{l,
    \bar{r}}^{k, \bar{h}} \tmop{Rm}_{j, \bar{s}}^{l, \bar{r}}\;,
  \end{eqnarray*}
  which implies the symmetry of the indices $k, s$ for $\partial^2_{j,
  \bar{k}} R_{h, \bar{s}}$. The symmetry of the indices $j, h$ follows by
  conjugation. Moreover at the space time point $(x_0, t_0)$ hold
  \begin{eqnarray*}
    \partial^2_{r, \bar{s}} \hspace{0.25em}  \mathcal{R}_{k, \bar{l}} & = & 2
    \hspace{0.25em} \partial^2_{r, \bar{s}} \hspace{0.25em} R'_{k, \bar{l}}
    \hspace{0.75em} - \hspace{0.75em} 2 \hspace{0.25em} R'_{k, \bar{p}}
    \hspace{0.25em} \partial^2_{r, \bar{s}} \hspace{0.25em}  \hat{\omega}_{p,
    \bar{l}} \\
    &  & \\
    & = & 2 \hspace{0.25em} \partial^2_{r, \bar{s}} \hspace{0.25em} R_{k,
    \bar{l}} \hspace{0.75em} + \hspace{0.75em} 2 \hspace{0.25em}
    \partial^4_{r, \bar{s}, k, \bar{l}} \hspace{0.25em}  \hat{f}
    \hspace{0.25em} \hspace{0.75em} - \hspace{0.75em} 2 \hspace{0.25em} R'_{k,
    \bar{p}} \hspace{0.25em} \partial^2_{r, \bar{s}} \hspace{0.25em} 
    \hat{\omega}_{p, \bar{l}}  \hspace{0.25em} .
  \end{eqnarray*}
  We infer
  \begin{eqnarray*}
    \partial^2_{r, \bar{s}} \hspace{0.25em}  \mathcal{R}_{k, \bar{l}} & = &
    \partial^2_{k, \bar{l}} \hspace{0.25em}  \mathcal{R}_{r, \bar{s}} \;\;+\;\;
    2\,\mathcal{R}_{r, \bar{p}}\; \partial^2_{k, \bar{l}} \hspace{0.25em} 
    \hat{\omega}_{p, \bar{s}} \;\;-\;\; 2\,\mathcal{R}_{k, \bar{p}} \;\partial^2_{r,
    \bar{s}} \hspace{0.25em}  \hat{\omega}_{p, \bar{l}}\;,
  \end{eqnarray*}
  and thus the required conclusion.
\end{proof}
From lemma \ref{Sim-Ric} we deduce the identity along the K\"ahler-Ricci flow
\begin{equation}
  \label{Lap-Rc} \Delta \hspace{0.25em}  \underline{\tmop{Ric}}^{\ast}
  \hspace{0.75em} = \hspace{0.75em}  \left( i \partial \overline{\partial}
  \tmop{Tr}_{_{\R}}  \underline{\tmop{Ric}}^{\ast}\;\; -\;\; 2\,
  \underline{\tmop{Ric}} \tmop{Rm}_{} \right)^{\ast} \hspace{0.75em} +
  \hspace{0.75em}  \underline{\tmop{Ric}}^{\ast}_{} \tmop{Ric}^{\ast}
  \hspace{0.75em} + \hspace{0.75em} \tmop{Ric}^{\ast}  \hspace{0.25em} 
  \underline{\tmop{Ric}}^{\ast} .
\end{equation}

{\bf{Part II.}} We remind first the local expression
\begin{eqnarray*}
  A & = & A_{k, \bar{l}} \;\, \bar{\zeta}_l^{\ast} \otimes \zeta_k
  \hspace{0.75em} + \hspace{0.75em} \overline{A_{k, \bar{l}}} \;\,
  \zeta_l^{\ast} \otimes \bar{\zeta}_k \hspace{0.25em},\\
  &  & \\
  A_{k, \bar{l}} & = & 2 \;\hat{\omega}^{p, \bar{k}}  \left[
  \partial^2_{\bar{l}, \bar{p}} \hspace{0.25em} \hat{f} \hspace{0.75em} -
  \hspace{0.75em} \partial_{\bar{l}} \hspace{0.25em}  \hat{\omega}_{j,
  \bar{p}}  \; \hat{\omega}^{r, \bar{j}} \hspace{0.25em}
  \partial_{\bar{r}} \hspace{0.25em} \hat{f} \right] \hspace{0.25em} .
\end{eqnarray*}
(See the subsection \ref{loc-expr-antH} in the appendix.) As in part I we infer the expressions at the space time
point $(x_0, t_0)$
\begin{eqnarray*}
  \Delta \hspace{0.25em} A & = & 4 \hspace{0.25em} \partial^2_{r, \bar{r}}
  \hspace{0.25em} A_{k, \bar{l}} \;\,  \bar{\zeta}_l^{\ast} \otimes
  \zeta_k \hspace{0.75em} + \hspace{0.75em} 4 \hspace{0.25em} \partial^2_{r,
  \bar{r}} \hspace{0.25em}  \overline{A_{k, \bar{l}}}  \;\,
  \zeta_l^{\ast} \otimes \bar{\zeta}_k\\
  &  & \\
  & + & 2 \hspace{0.25em} A_{k, \bar{l}} \;\, \nabla_{r}
  \hspace{0.25em} \nabla_{\bar{r}} \hspace{0.25em}  \bar{\zeta}^{\ast}_l
  \otimes \zeta_k \hspace{0.75em} + \hspace{0.75em} 2 \hspace{0.25em}
  \overline{A_{k, \bar{l}}} \;\, \zeta_l^{\ast} \otimes \nabla_{r} \hspace{0.25em} \nabla_{\bar{r}}  \hspace{0.25em} \bar{\zeta}_k\\
  &  & \\
  & + & 2 \hspace{0.25em} A_{k, \bar{l}} \;\,
  \bar{\zeta}_l^{\ast} \otimes \nabla_{\bar{r}} \hspace{0.25em} \nabla_{r} \hspace{0.25em} \zeta_k \hspace{0.75em} + \hspace{0.75em} 2
  \hspace{0.25em} \overline{A_{k, \bar{l}}} \;\, \nabla_{\bar{r}} \hspace{0.25em} \nabla_{r} \hspace{0.25em} \zeta^{\ast}_l
  \otimes \bar{\zeta}_k\\
  &  & \\
  & = & 4 \hspace{0.25em} \partial^2_{r, \bar{r}} \hspace{0.25em} A_{k,
  \bar{l}} \;\,  \bar{\zeta}_l^{\ast} \otimes \zeta_k
  \hspace{0.75em} + \hspace{0.75em} 4 \hspace{0.25em} \partial^2_{r, \bar{r}}
  \hspace{0.25em}  \overline{A_{k, \bar{l}}} \;\, \zeta_l^{\ast}
  \otimes \bar{\zeta}_k \\
  &  & \\
  & + & 2 \left( A_{k, \bar{p}} \hspace{0.25em} R_{p, \bar{l}}
  \hspace{0.75em} - \hspace{0.75em} A_{p, \bar{l}} \hspace{0.25em} R_{p,
  \bar{k}} \right)  \bar{\zeta}_l^{\ast} \otimes \zeta_k\\
  &  & \\
  & + & 2 \left( \hspace{0.25em} \overline{A_{k, \bar{p}}} \hspace{0.25em}
  R_{l, \bar{p}} \hspace{0.75em} - \hspace{0.75em}  \overline{A_{p, \bar{l}}}
  \hspace{0.25em} R_{k, \bar{p}} \right) \zeta_l^{\ast} \otimes \bar{\zeta}_k
  \hspace{0.25em} .
\end{eqnarray*}
Expanding the second order derivative we infer for all indices $k, l$ the
expression
\[ \partial^2_{r, \bar{r}} \hspace{0.25em} A_{k, \bar{l}} \hspace{0.75em} =
   \hspace{0.75em} 2 \hspace{0.25em} \partial^4_{r, \bar{r}, \bar{k}, \bar{l}}
   \hspace{0.25em}  \hat{f} \hspace{0.75em} + \hspace{0.75em} 2
   \hspace{0.25em} \partial_{\bar{l}}\; R_{p, \bar{k}} \hspace{0.25em}
   \partial_{\bar{p}}  \;\hat{f} \hspace{0.75em} + \hspace{0.75em} 2
   \hspace{0.25em} R_{p, \bar{k}} \hspace{0.25em} \partial^2_{\bar{l},
   \bar{p}}  \;\hat{f} \hspace{0.25em} . \]
By plunging this in the previous expression and canceling the adequate therms
we obtain
\begin{eqnarray*}
  \Delta \hspace{0.25em} A & = & \bar{\partial} \hspace{0.25em} \nabla \Delta
  \hat{f} \hspace{0.75em} + \hspace{0.75em} 2 \hspace{0.25em}
  \overline{\partial}_{_E} \tmop{Ric}^{\ast} \cdot \nabla \hat{f}\\
  &  & \\
  & + & 4 \hspace{0.25em} \partial^2_{\bar{l}, \bar{p}}  \;\hat{f}
  \;\, R_{p, \bar{k}} \;\, \bar{\zeta}_l^{\ast} \otimes
  \zeta_k \hspace{0.75em} + \hspace{0.75em} 4 \hspace{0.25em} R_{k, \bar{p}}
  \;\, \partial^2_{p, l}  \;\hat{f} \;\, \zeta_l^{\ast}
  \otimes \bar{\zeta}_k\\
  &  & \\
  & + & 4 \hspace{0.25em} \partial^2_{\bar{k}, \bar{p}}  \;\hat{f}
  \;\, R_{p, \bar{l}} \;\,  \bar{\zeta}_l^{\ast} \otimes
  \zeta_k \hspace{0.75em} + \hspace{0.75em} 4 \hspace{0.25em} R_{l, \bar{p}}
  \;\, \partial^2_{p, k}  \;\hat{f} \;\, \zeta_l^{\ast}
  \otimes \bar{\zeta}_k \hspace{0.25em},
\end{eqnarray*}
and thus the identity
\begin{equation}
  \label{Lap-A} \Delta \hspace{0.25em} A = \overline{\partial}
  \hspace{0.25em} \nabla \Delta \hat{f} \hspace{0.75em} + \hspace{0.75em} 2
  \hspace{0.25em} \overline{\partial}_{_E} \tmop{Ric}^{\ast} \cdot \nabla
  \hat{f} \hspace{0.75em} + \hspace{0.75em} \tmop{Ric}^{\ast} A
  \hspace{0.75em} + \hspace{0.75em} A \tmop{Ric}^{\ast} \hspace{0.25em} .
\end{equation}
In conclusion the identities (\ref{Lap-Rc}), (\ref{Lap-A}), (\ref{Sym-Tr-Rc})
(at the level of scalar products) and (\ref{Sym-Tr-A}) imply
\begin{eqnarray*}
  \Delta \left( | \alpha^{\ast} |^2 \hspace{0.75em} + \hspace{0.75em} |A|^2
  \right) & = & 2 | \nabla \alpha^{\ast} |^2 \hspace{0.75em} + \hspace{0.75em}
  2 | \nabla A|^2\\
  &  & \\
  & - & 2 \left\langle \alpha^{\ast}, \Delta \hspace{0.25em} 
  \underline{\tmop{Ric}}^{\ast}  \right\rangle \hspace{0.75em} +
  \hspace{0.75em} 2 \left\langle A, \Delta A \right\rangle\\
  &  & \\
  & = & 2 | \nabla \hspace{0.25em} \alpha^{\ast} |^2 \hspace{0.75em} +
  \hspace{0.75em} 2 | \nabla A|^2\\
  &  & \\
  & - & 2 \left\langle \alpha^{\ast}, \left[ i \partial \overline{\partial}
  \left( \tmop{Scal} \hspace{0.75em} + \hspace{0.75em} \Delta \hat{f} \right)
  \right]^{\ast} \right\rangle\\
  &  & \\
  & + & 4 \left\langle \alpha^{\ast}, \left( \underline{\tmop{Ric}}
  \tmop{Rm}_{} \right)^{\ast} \hspace{0.75em} - \hspace{0.75em}
  \tmop{Ric}^{\ast}  \hspace{0.25em}  \underline{\tmop{Ric}}^{\ast}
  \right\rangle\\
  &  & \\
  & + & 2 \left\langle A, \overline{\partial} \hspace{0.25em} \nabla \Delta
  \hat{f} \hspace{0.75em} + \hspace{0.75em} 2 \hspace{0.25em}
  \overline{\partial}_{_E} \tmop{Ric}^{\ast} \cdot \nabla \hat{f}
  \right\rangle\\
  &  & \\
  & + & 4 \left\langle A^2, \tmop{Ric}^{\ast}  \right\rangle \hspace{0.25em}
  .
\end{eqnarray*}

\subsection{Evolution of the squared norms}

The previous expression of the Laplacian of the squared norms combined with the
formulas (\ref{der-Nr-Rc}) and (\ref{Ev-Nrm-A}) provides the identity
\begin{eqnarray*}
  \Box \left( | \alpha^{\ast} |^2 \hspace{0.75em} + \hspace{0.75em} |A|^2
  \right) & = & 2 | \nabla \alpha^{\ast} |^2 \hspace{0.75em} + \hspace{0.75em}
  2 | \nabla A|^2\\
  &  & \\
  & - & 2 \left\langle \alpha^{\ast}, \left[ i \partial \overline{\partial}
  \left( 2 \hspace{0.25em} \Delta \hat{f} \hspace{0.75em} - \hspace{0.75em} |
  \nabla \hat{f} |^2 \hspace{0.75em} + \hspace{0.75em} \tmop{Scal} \right)
  \right]^{\ast} \right\rangle\\
  &  & \\
  & + & 2 \left\langle A, \overline{\partial} \hspace{0.25em} \nabla \left( 2
  \hspace{0.25em} \Delta \hat{f} \hspace{0.75em} - \hspace{0.75em} | \nabla
  \hat{f} |^2 \hspace{0.75em} + \hspace{0.75em} \tmop{Scal} \hspace{0.75em} +
  \hspace{0.75em} 2 \hspace{0.25em} \hat{f} \right) \right\rangle\\
  &  & \\
  & + & 4 \left\langle \alpha^{\ast}, \left( \underline{\tmop{Ric}}^{\ast}_{}
  \tmop{Rm} \right)^{\ast}  \right\rangle \hspace{0.75em} - \hspace{0.75em} 4
  \left\langle \alpha^{\ast}, \underline{\tmop{Ric}}^{\ast}_{}  \right\rangle 
  \hspace{0.25em} .
\end{eqnarray*}
Thus if we set 
$$
2\; \hat{H} \;\; : = \;\; 2\; H ( \hat{g},
\hat{f}) \;\; = \;\; \Box\, \hat{f} \;\; +\;\;
2 \; \hat{f}\;,
$$
we infer the evolution formula
\begin{eqnarray}
  \Box \left( | \alpha^{\ast} |^2 \hspace{0.75em} + \hspace{0.75em} |A|^2
  \right) & = & 2 | \nabla \alpha^{\ast} |^2 \hspace{0.75em} + \hspace{0.75em}
  2 | \nabla A|^2 \nonumber\\
  &  &  \nonumber\\
  & - & 4 \left\langle \alpha^{\ast}, (i \partial \overline{\partial} 
  \hat{H})^{\ast} \right\rangle \hspace{0.75em} + \hspace{0.75em} 4
  \left\langle A, \overline{\partial} \hspace{0.25em} \nabla \hat{H}
  \right\rangle \nonumber\\
  &  &  \nonumber\\
  & - & 4 \left\langle \alpha \tmop{Rm}, \alpha \right\rangle \hspace{0.25em}
  .  \label{ev-Nr2}
\end{eqnarray}

\subsection{Integration by parts along the K\"ahler-Ricci flow}

This last step is the key part of the proof of the second variation of
Perelman's $\mathcal{W}$-functional along the K\"ahler-Ricci flow. Deriving
the identity
\[ 0 \hspace{0.75em} = \hspace{0.75em} \int_X \left( \Delta \,\hat{H}
   \hspace{0.75em} - \hspace{0.75em} \nabla \hat{H} \cdot \nabla \hat{f}
   \hspace{0.25em} \right) e^{- \hat{f}} d V_{\hat{g}} \hspace{0.25em}, \]
we obtain
\begin{equation}
  \label{der-W-k} 0 = \int_X \Box \left( \Delta \,\hat{H} \hspace{0.75em} -
  \hspace{0.75em} \nabla \hat{H} \cdot \nabla \hat{f} \hspace{0.25em} \right)
  e^{- \hat{f}} d V_{\hat{g}}  \hspace{0.25em} ,
\end{equation}
thanks to (\ref{BOX-pot-EVOL}). We expand now the integrand therm in (\ref{der-W-k}). We observe first the identity
\begin{eqnarray*}
  \Box \,\Delta \,\hat{H} & = & \Delta \, \Box \,
  \hat{H} \hspace{0.75em} + \hspace{0.75em} 2 \hspace{0.25em} \Delta\, \hat{H}
  \hspace{0.75em} - \hspace{0.75em} 2 \left\langle \tmop{Ric}, i \partial
  \overline{\partial}  \hat{H} \right\rangle\\
  &  & \\
  & = & 2 \hspace{0.25em} \Delta \,H_2 \hspace{0.75em} - \hspace{0.75em} 2
  \left\langle \tmop{Ric}, i \partial \overline{\partial}  \hat{H}
  \right\rangle \hspace{0.25em},
\end{eqnarray*}
where 
$$
2 \;H_2 \;\;: =\;\;\Box \; \hat{H} \;\;+\;\; 2\; \hat{H}\;.
$$ 
(For more details on
this type of computations see the proof of the formula (\ref{ev-H}) in the subsection \ref{Per-first-varKRF} of the
appendix). We compute next the heat therm 
$$
\Box \hspace{0.25em} (\nabla \hat{H}
\cdot \nabla \hat{f})\;.
$$
We observe first the expression of the time derivative
\begin{eqnarray*}
  \frac{\partial}{\partial t}  \left( \nabla \hat{H} \cdot \nabla \hat{f}
  \right) & = & \frac{\partial}{\partial t}  \left( d \hat{H} \cdot \nabla
  \hat{f} \right)\\
  &  & \\
  & = & d \left( \frac{\partial}{\partial t}  \hspace{0.25em} \hat{H} \right)
  \cdot \nabla \hat{f} \hspace{0.75em} + \hspace{0.75em} d \hat{H} \cdot
  \frac{\partial}{\partial t} \hspace{0.25em} \nabla \hat{f}\\
  &  & \\
  & = & \nabla \left( \frac{\partial}{\partial t} \hspace{0.25em}  \hat{H}
  \right) \cdot \nabla \hat{f} \hspace{0.75em} + \hspace{0.75em} \nabla
  \hat{H} \cdot \nabla \left( \frac{\partial}{\partial t}  \hspace{0.25em}
  \hat{f} \hspace{0.75em} - \hspace{0.75em} \hat{f} \right)\\
  &  & \\
  & + & \tmop{Ric} \left( \nabla \hat{f}, J_0 \nabla \hat{H} \right) 
  \hspace{0.25em} .
\end{eqnarray*}
We expand the Laplacian
\begin{eqnarray*}
  \Delta \left( \nabla \hat{H} \cdot \nabla \hat{f} \right) & = & 8
  \hspace{0.25em} \partial^2_{r, \bar{r}}  \left[ \hat{\omega}^{k, \bar{l}}
  \left( \hat{H}_l \hspace{0.25em}  \hat{f}_{\bar{k}} \hspace{0.75em} +
  \hspace{0.75em}  \hat{f}_l  \hspace{0.25em} \hat{H}_{\bar{k}} \right)
  \right]\\
  &  & \\
  & = & - \hspace{0.25em} 8 \hspace{0.25em} \partial^2_{r, \bar{r}}
  \hspace{0.25em} \hat{\omega}_{k, \bar{l}} \hspace{0.25em} \left(
  \hat{H}_l \hspace{0.25em}  \hat{f}_{\bar{k}} \hspace{0.75em} +
  \hspace{0.75em} \hat{f}_l \hspace{0.25em}  \hat{H}_{\bar{k}} \right)
  \hspace{0.75em} \\
  &  & \\
  & + & 8 \hspace{0.25em} \partial^2_{r, \bar{r}} \hspace{0.25em}  \left(
  \hat{H}_l \hspace{0.25em} \hat{f}_{\bar{k}} \hspace{0.75em} +
  \hspace{0.75em} \hat{f}_l \hspace{0.25em}  \hat{H}_{\bar{k}} \right)
\end{eqnarray*}
\begin{eqnarray*}
& = & 8 \hspace{0.25em} R_{k, \bar{l}} \hspace{0.25em}  \left( \hat{H}_l
  \hspace{0.25em}  \hat{f}_{\bar{k}} \hspace{0.75em} + \hspace{0.75em} 
  \hat{f}_l \hspace{0.25em}  \hat{H}_{\bar{k}} \right)\\
  &  & \\
  & + & 2 \partial_k \Delta \hat{H} \hspace{0.25em} \hat{f}_{\bar{k}}
  \hspace{0.75em} + \hspace{0.75em} 8 \hat{H}_{k, \bar{l}} \hspace{0.25em} 
  \hat{f}_{\bar{k}, l} \hspace{0.75em} \\
  &  & \\
  & + & 8 \hat{H}_{k, l} \hspace{0.25em}  \hat{f}_{\bar{k}, \bar{l}}
  \hspace{0.75em} + \hspace{0.75em} 2 \hat{H}_k \hspace{0.25em}
  \partial_{\bar{k}} \Delta \hat{f}\\
  &  & \\
  & + & 2 \partial_k \Delta \hat{f} \hspace{0.25em}  \hat{H}_{\bar{k}}
  \hspace{0.75em} + \hspace{0.75em} 8 \hat{f}_{k, \bar{l}} \hspace{0.25em}
  \hat{H}_{\bar{k}, l} \hspace{0.75em}\\
  &  & \\
  & + & 8 \hat{f}_{k, l} \hspace{0.25em} \hat{H}_{\bar{k}, \bar{l}}
  \hspace{0.75em} + \hspace{0.75em} 2 \hat{f}_k \hspace{0.25em}
  \partial_{\bar{k}} \Delta \hat{H}\\
  &  & \\
  & = & 2 \tmop{Ric} (\nabla \hat{H}, J_0 \nabla \hat{f}) \hspace{0.75em} +
  \hspace{0.75em} \nabla \Delta \hat{H} \cdot \nabla \hat{f} \hspace{0.75em} +
  \hspace{0.75em} \nabla \hat{H} \cdot \nabla \Delta \hat{f}\\
  &  & \\
  & + & 2 \left\langle i \partial \overline{\partial}  \hat{H}, i \partial
  \overline{\partial}  \hat{f} \right\rangle \hspace{0.75em} + \hspace{0.75em}
  2 \left\langle \overline{\partial} \hspace{0.25em} \nabla \hat{H},
  \overline{\partial} \hspace{0.25em} \nabla \hat{f} \right\rangle 
  \hspace{0.25em} .
\end{eqnarray*}
We infer the equality
\begin{eqnarray*}
  \Box \left( \nabla \hat{H} \cdot \nabla \hat{f} \right) & = & \nabla
  \hspace{0.25em} \Box \hspace{0.25em} \hat{H} \cdot \nabla \hat{f}
  \hspace{0.75em} + \hspace{0.75em} \nabla \hat{H} \cdot \nabla
  \hspace{0.25em} \Box \hspace{0.25em}  \hat{f} \hspace{0.75em} +
  \hspace{0.75em} 2 \hspace{0.25em} \nabla \hat{H} \cdot \nabla \hat{f}\\
  &  & \\
  & + & 2 \left\langle i \partial \overline{\partial}  \hat{H}, i \partial
  \overline{\partial}  \hat{f} \right\rangle \hspace{0.75em} + \hspace{0.75em}
  2 \left\langle \overline{\partial} \hspace{0.25em} \nabla \hat{H},
  \overline{\partial} \hspace{0.25em} \nabla \hat{f} \right\rangle\\
  &  & \\
  & = & 2 \hspace{0.25em} \nabla H_2 \cdot \nabla \hat{f} \hspace{0.75em} +
  \hspace{0.75em} 2 \hspace{0.25em} | \nabla \hat{H} |^2 \hspace{0.75em} -
  \hspace{0.75em} 2 \hspace{0.25em} \nabla \hat{H} \cdot \nabla \hat{f}\\
  &  & \\
  & + & 2 \left\langle i \partial \overline{\partial}  \hat{H}, i \partial
  \overline{\partial}  \hat{f} \right\rangle \hspace{0.75em} + \hspace{0.75em}
  2 \left\langle \overline{\partial} \hspace{0.25em} \nabla \hat{H},
  \overline{\partial} \hspace{0.25em} \nabla \hat{f} \right\rangle
  \hspace{0.25em} .
\end{eqnarray*}
Finally we obtain the identity
\begin{eqnarray*}
  \Box \left( \Delta \,\hat{H}\;\; -\;\; \nabla \hat{H} \cdot \nabla \hat{f} \right) & =
  & 2 \hspace{0.25em} \Delta\, H_2 \hspace{0.75em} - \hspace{0.75em} 2
  \hspace{0.25em} \nabla H_2 \cdot \nabla \hat{f}\\
  &  & \\
  & + & 2 \hspace{0.25em} \nabla \hat{H} \cdot \nabla \hat{f} \hspace{0.75em}
  - \hspace{0.75em} 2 \hspace{0.25em} | \nabla \hat{H} |^2\\
  &  & \\
  & - & 2 \left\langle \tmop{Ric} \hspace{0.75em} + \hspace{0.75em} i
  \partial \overline{\partial}  \hat{f}, i \partial \overline{\partial} 
  \hat{H} \right\rangle \hspace{0.75em} - \hspace{0.75em} 2 \left\langle A,
  \overline{\partial} \hspace{0.25em} \nabla \hat{H} \right\rangle
  \hspace{0.25em} .
\end{eqnarray*}
By adding and subtracting the therm $2 \Delta \hat{H}$ we obtain the
evolution formula
\begin{eqnarray*}
  \Box \left( \Delta \,\hat{H} \;\;-\;\; \nabla \hat{H} \cdot \nabla \hat{f} \right) & =
  & 2 \hspace{0.25em} \Delta\, H_2 \hspace{0.75em} - \hspace{0.75em} 2
  \hspace{0.25em} \nabla H_2 \cdot \nabla \hat{f}\\
  &  & \\
  & + & 2 \hspace{0.25em} \nabla \hat{H} \cdot \nabla \hat{f} \hspace{0.75em}
  - \hspace{0.75em} 2 \,\Delta \,\hat{H} \hspace{0.75em} - \hspace{0.75em} 2
  \hspace{0.25em} | \nabla \hat{H} |^2\\
  &  & \\
  & + & 2 \left\langle \alpha^{\ast}, (i \partial \overline{\partial} 
  \hat{H})^{\ast} \right\rangle \hspace{0.75em} - \hspace{0.75em} 2
  \left\langle A, \overline{\partial} \hspace{0.25em} \nabla \hat{H}
  \right\rangle \hspace{0.25em},
\end{eqnarray*}
which combined with the identity (\ref{der-W-k}) implies the formula
\begin{equation}
  \label{Magic-Heat} \int_X | \nabla \hat{H} |^2  \hspace{0.25em} e^{-
  \hat{f}} d V_{\hat{g}} \hspace{0.75em} = \hspace{0.75em} \int_X \left[
  \left\langle \alpha^{\ast}, (i \partial \overline{\partial}  \hat{H})^{\ast}
  \right\rangle \hspace{0.75em} - \hspace{0.75em}  \left\langle A,
  \overline{\partial} \hspace{0.25em} \nabla \hat{H} \right\rangle \right]
  e^{- \hat{f}} d V_{\hat{g}}  \hspace{0.25em},
\end{equation}
via integration by parts. Combining the equality (\ref{Magic-Heat}) with the identity
(\ref{pre-SEcW}) and with the evolution formula (\ref{ev-Nr2}) we infer the second variation
identity
\begin{eqnarray*}
  \ddot{\mathcal{W}} & = & \int_X \left[ 2 \left\langle \alpha \tmop{Rm},
  \alpha \right\rangle \hspace{0.75em} + \hspace{0.75em} 2 \hspace{0.25em} |
  \nabla \hat{H} |^2 \hspace{0.75em} - \hspace{0.75em} | \nabla \alpha^{\ast}
  |^2 \hspace{0.75em} - \hspace{0.75em} | \nabla A|^2 \right] e^{- \hat{f}} d
  V_{\hat{g}}  \hspace{0.25em} .
\end{eqnarray*}
By using the diffeomorphisms invariance of the integrals on the r.h.s. we infer
the conclusion of theorem \ref{W-secVar}.{\hspace*{\fill}}
\section{Interpretation of formula (\ref{Magic-Heat}) along the \\modified
K\"ahler-Ricci flow}\label{modKRF-Magic-Heat}

We remind that the $\Omega$-Bakry-Emery-Ricci tensor of a metric $g$ is
defined by the formula
\[ \tmop{Ric}_g (\Omega) \hspace{0.75em} : = \hspace{0.75em} \tmop{Ric} (g)
   \hspace{0.75em} + \hspace{0.75em} \nabla_g \,d\, \log \frac{dV_g}{\Omega} \;. \]
We denote by $\tmop{Ric}^{\ast}_g (\Omega) \assign g^{- 1} \tmop{Ric}_g
(\Omega)$ the endomorphism section associated to the
$\Omega$-Bakry-Emery-Ricci tensor. We set $f \;\assign\; \log \frac{dV_g}{\Omega}$
and we observe that the operator
\begin{eqnarray*}
  \nabla^{\ast_{_{\Omega}}}_g & : = & e^f \nabla^{\ast}_g  \left( e^{- f}
  \bullet \right) \;\;=\;\; \nabla^{\ast}_g \;\,+\;\, \nabla_g\, f \;\neg\;,
\end{eqnarray*}
is the formal adjoint of $\nabla_g$ with respect to the scalar product $\int_X
\left\langle \cdot, \cdot \right\rangle_g \Omega$. \ \

We show now Perelman's version of the twice contracted differential Bianchi
identity (see \cite{Per})
\begin{equation}
  \label{AdRic-H} \nabla^{\ast_{_{\Omega}}}_g \tmop{Ric}^{\ast}_g (\Omega) \;\;=\;\;
  \nabla_g \,(f\;\, -\;\, H)\;,
\end{equation}
In fact expanding the l.h.s we obtain
\begin{eqnarray*}
  2 \,\nabla^{\ast_{_{\Omega}}}_g \tmop{Ric}^{\ast}_g (\Omega) & = & 2\,
  \nabla^{\ast}_g \tmop{Ric}^{\ast}_g (\Omega) \;\,+\;\, 2\, \tmop{Ric}^{\ast}_g
  (\Omega) \cdot \nabla_g \,f\\
  &  & \\
  & = & 2 \,\nabla^{\ast}_g \tmop{Ric}^{\ast}_g \;\,+\;\, 2\, \nabla^{\ast}_g \nabla^2_g
  f 
\\
\\
&+& 2\, \tmop{Ric}^{\ast}_g \cdot \nabla_g \,f \;\,+\;\, 2\, \nabla^2_g \,f \cdot \nabla_g\,
  f\\
  &  & \\
  & = & - \;\,\nabla_g \tmop{Scal}_g \;\,-\;\, 2\, \Delta_g \nabla_g \,f 
\\
\\
&+& 2\,
  \tmop{Ric}^{\ast}_g \cdot \nabla_g \,f \;\,+\;\, \nabla_g | \nabla_g \,f|_g^2\;,
\end{eqnarray*}
thanks to the twice contracted Bianchi identity. Then formula (\ref{AdRic-H})
follows from the identity
\begin{eqnarray*}
  \nabla_g \,\Delta_g \,f & = & \Delta_g \nabla_g \,f \;\,-\;\, \tmop{Ric}^{\ast}_g \cdot
  \nabla_g \,f \;.
\end{eqnarray*}
We remind now that any smooth volume form $\Omega > 0$ over a complex manifold
$(X, J)$ of complex dimension $n$ induces a hermitian metric $h_{\Omega}$
over the canonical bundle $K_{_{X, J}} \assign \Lambda_{_J}^{n, 0}
T^{\ast}_{_X} $ given by the formula
\[ h_{\Omega} (\alpha, \beta) \;\;\assign\;\; \frac{n!\, i^{n^2} \alpha \wedge
   \overline{\beta} }{\Omega} \;. \]
By abuse of notations we will denote by $\Omega^{- 1}$ the metric $h_{\Omega}
.$ The dual metric $h_{\Omega}^{\ast}$ on the anti canonical bundle $K^{-
1}_{_{X, J}} = \Lambda_{_J}^{n, 0} T_{_X}$ is given by the formula
\[ h_{\Omega}^{\ast} (\xi, \eta) \;\;=\;\; (- i)^{n^2} \Omega \left( \xi_{},
   \bar{\eta} \right) / n! \;. \]
Abusing notations again, we denote by $\Omega$ the dual metric
$h_{\Omega}^{\ast}$. We define the $\Omega$-Ricci form
$$
\tmop{Ric}_{_J} \left( \Omega \right) \;\;\assign\;\; i\,\mathcal{C}_{\Omega}  \big(
   K^{- 1}_{_{X, J}} \big) \;\;=\;\; - \;\,i\,\mathcal{C}_{\Omega^{- 1}} \big( K_{_{X,
   J}} \big) \;, 
$$
where $\mathcal{C}_h(L)$ denotes the Chern curvature of a hermitian line bundle $(L,h)$.
In particular $\tmop{Ric}_{_J} (\omega) = \tmop{Ric}_{_J} (\omega^n)$. We
remind also that for any $J$-invariant K\"ahler metric $g$ the associated
symplectic form $\omega \assign g J$ satisfies the elementary identity $\tmop{Ric} (g)  =  - \; \tmop{Ric}_{_J} (\omega)\, J$.
Moreover for all twice differentiable function $u$ hold the identity
\begin{eqnarray*}
  \nabla_g \,d\,u & = & - \;\, \big(i\, \partial_{_J} \overline{\partial}_{_J}
  u \big)\, J \;\, + \;\, g \,\overline{\partial}_{_{T_{X,
  J}}} \nabla_g \,u\; .
\end{eqnarray*}
We infer the decomposition identity
\begin{eqnarray*}
  \tmop{Ric}_g (\Omega) \;\; = \;\; - \hspace{0.25em} \tmop{Ric}_{_J} (\Omega) \,J
  \;\, + \;\, g \,\overline{\partial}_{_{T_{X, J}}}
  \nabla_g \log \frac{dV_g}{\Omega}\;,
\end{eqnarray*}
and thus the identity
\begin{equation}
  \label{dec-end-Ric} \tmop{Ric}^{\ast}_{g_{_{}}} (\Omega) \;\;=\;\;
  \tmop{Ric}^{\ast}_{_J} (\Omega) \;\,+\;\, \overline{\partial}_{_{T_{X, J}}}
  \nabla_g \log \frac{d V_g}{\Omega} \;,
\end{equation}
where $\tmop{Ric}^{\ast}_{_J} (\Omega) : = \omega^{- 1} \tmop{Ric}_{_J} (\Omega)$.
We will denote by $\left\langle \cdot, \cdot \right\rangle_{\omega} $the
hermitian product on $T_X$-valued forms induced by the hermitian metric on
$T_{X, J}$. 
The formal adjoint of the $\partial^g_{_{T_{X, J}}}$-operator with
respect to the $L^2$-hermitian product $\int_X \left\langle \cdot, \cdot
\right\rangle_{\omega} \Omega$, is the operator
\begin{eqnarray*}
  \partial^{\ast_{g, \Omega}}_{_{T_{X, J}}}  \;\;: =\;\;  e^f\,
  \partial^{\ast_g}_{_{T_{X, J}}} \left( e^{- f} \bullet \right) \;.
\end{eqnarray*}
In a similar way the formal adjoint of the $\overline{\partial}_{_{T_{X,
J}}}$-operator with respect to the $L^2$-hermitian product $\int_X \left\langle
\cdot, \cdot \right\rangle_{\omega} \Omega$, is the operator
\begin{eqnarray*}
  \overline{\partial}^{\ast_{g, \Omega}}_{_{T_{X, J}}} \;\;: = \;\; e^f\,
  \overline{\partial}^{\ast_g}_{_{T_{X, J}}} \left( e^{- f} \bullet \right) \;.
\end{eqnarray*}
With this notations hold the decomposition formula
\begin{eqnarray*}
  p \,\nabla^{\ast_{_{\Omega}}}_g \;\; = \;\; \partial^{\ast_{g, \Omega}}_{_{T_{X,
  J}}} \;\,+\;\, \overline{\partial}^{\ast_{g, \Omega}}_{_{T_{X, J}}}\;,
\end{eqnarray*}
at the level of $T_X$-valued $p$-forms. (See the appendix in \cite{Pal}). We observe
also that the identity $d \tmop{Ric}_{_J} (\Omega) = 0$ is equivalent to the
identity $\partial_{_J} \tmop{Ric}_{_J} (\Omega) = 0,$ which in its turn is
equivalent to the identity
\begin{eqnarray*}
  \partial_{_{T_{X, J}}}^g \tmop{Ric}^{\ast}_{_J} (\Omega) \;\; = \;\; 0\; .
\end{eqnarray*}
We deduce
\begin{eqnarray*}
  \nabla^{\ast_{_{\Omega}}}_g \tmop{Ric}^{\ast}_{_{J_t}} (\Omega) & = &
  \partial^{\ast_{g_t, \Omega}}_{_{T_{X, J_t}}} \tmop{Ric}^{\ast}_{_{J_t}}
  (\Omega) \;.
\end{eqnarray*}
thanks to a basic K\"ahler identity. We observe now that by the diffeomorphisms invariance of the
integrals formula (\ref{Magic-Heat}) writes as
\begin{eqnarray*}
  \int_X | \nabla_t H_t |_t^2 \,\Omega & = &  -\;\, \int_X \left\langle
  \tmop{Ric}^{\ast}_{_{J_t}} (\Omega) -\mathbbm{I}_{_{T_X}},
  \partial^{g_t}_{_{T_{X, J_t}}} \nabla_t H_t \right\rangle_t 
\\
\\
&-&
  \int_X\left\langle \overline{\partial}_{_{T_{X, J_t}}} \nabla_t\,
  f_t\,,\, \overline{\partial}_{_{T_{X, J}}} \nabla_t H_t \right\rangle_t\Omega\;,
\end{eqnarray*}
along the modified K\"ahler-Ricci flow. We show now this formula by using the
previous considerations. Using the comparison of norms on $T_{_{X, J}}$-valued forms in the subsection \ref{comp-CX-norms} of the appendix we
expand the integral therm
\begin{eqnarray*}
  & - & \int_X \left[ \left\langle \tmop{Ric}^{\ast}_{_{J_t}} (\Omega)\,,
  \partial^{g_t}_{_{T_{X, J_t}}} \nabla_t H_t \right\rangle_t \;\,
  + \;\, \left\langle \overline{\partial}_{_{T_{X, J_t}}} \nabla_t\,
  f_t\,, \overline{\partial}_{_{T_{X, J}}} \nabla_t H_t \right\rangle_t \,\right]
  \Omega\\
  &  & \\
  & = &  - \;\,\frac{1}{2} \,\int_X \left[ \left\langle \tmop{Ric}^{\ast}_{_{J_t}}
  (\Omega)\,, \partial^{g_t}_{_{T_{X, J_t}}} \nabla_t H_t
  \right\rangle_{\omega_t} 
\;\,+ \;\,
 \left\langle
  \partial^{g_t}_{_{T_{X, J_t}}} \nabla_t H_t\,, \tmop{Ric}^{\ast}_{_{J_t}}
  (\Omega) \right\rangle_{\omega_t} \right] \Omega\\
  &  & \\
  & - & \frac{1}{2} \,\int_X \left[ \left\langle \overline{\partial}_{_{T_{X,
  J_t}}} \nabla_t \,f_t\,, \overline{\partial}_{_{T_{X, J}}} \nabla_t H_t
  \right\rangle_{\omega_t} 
\;\, + \;\,  \left\langle
  \overline{\partial}_{_{T_{X, J_t}}} \nabla_t H_t\,,
  \overline{\partial}_{_{T_{X, J}}} \nabla_t \,f_t \right\rangle_{\omega_t}
  \right] \Omega\\
  &  & \\
  & = &  - \;\,\frac{1}{2} \,\int_X \left[ \left\langle \partial^{\ast_{g_t,
  \Omega}}_{_{T_{X, J_t}}} \tmop{Ric}^{\ast}_{_{J_t}} (\Omega)\,, \nabla_t H_t
  \right\rangle_{\omega_t} 
\;\,+ \;\,
  \left\langle
  \nabla_t H_t\,, \partial^{\ast_{g_t, \Omega}}_{_{T_{X, J_t}}}
  \tmop{Ric}^{\ast}_{_{J_t}} (\Omega) \right\rangle_{\omega_t} \right]
  \Omega\\
  &  & \\
  & - & \frac{1}{2} \int_X \left[ \left\langle
  \overline{\partial}^{\ast_{g_t, \Omega}}_{_{T_{X, J_t}}}
  \overline{\partial}_{_{T_{X, J_t}}} \nabla_t \,f_t\,, \nabla_t H_t
  \right\rangle_{\omega_t} \;\, + \;\,  \left\langle
  \nabla_t H_t\,, \overline{\partial}^{\ast_{g_t, \Omega}}_{_{T_{X, J_t}}}
  \overline{\partial}_{_{T_{X, J_t}}} \nabla_t \,f_t \right\rangle_{\omega_t}
  \right] \Omega\\
  &  & \\
  & = &  - \;\,\int_X \left[ \left\langle \partial^{\ast_{g_t, \Omega}}_{_{T_{X,
  J_t}}} \tmop{Ric}^{\ast}_{_{J_t}} (\Omega)\,, \nabla_t H_t \right\rangle_t
  \;\,+ \;\,  \left\langle
  \overline{\partial}^{\ast_{g_t, \Omega}}_{_{T_{X, J_t}}}
  \overline{\partial}_{_{T_{X, J_t}}} \nabla_t \,f_t\,, \nabla_t H_t
  \right\rangle_t \right] \Omega\\
  &  & \\
  & = &  - \;\,\int_X \left[ \left\langle \nabla^{\ast_{_{\Omega}}}_g
  \tmop{Ric}^{\ast}_{_{J_t}} (\Omega)\,, \nabla_t H_t \right\rangle_t
  \;\,+\;\, 
 \left\langle \nabla^{\ast_{_{\Omega}}}_g
  \overline{\partial}_{_{T_{X, J_t}}} \nabla_t\, f_t\,, \nabla_t H_t
  \right\rangle_t \right] \Omega\\
  &  & \\
  & = & - \;\,\int_X \left\langle \nabla^{\ast_{_{\Omega}}}_g \tmop{Ric}^{\ast}_g
  (\Omega)\,, \nabla_t H_t \right\rangle_t \Omega\\
  &  & \\
  & = & \int_X \Big[ | \nabla_t H_t |_t^2 \;\, -\;\, \big\langle
  \nabla_t \,f_t\,, \nabla_t H_t \big\rangle_t \Big] \Omega\;,
\end{eqnarray*}
thanks to basic K\"ahler identities and thanks to formulas (\ref{dec-end-Ric})
and (\ref{AdRic-H}). The required formula follows from the trivial identities.
\begin{eqnarray*}
  \int_X \big\langle \nabla_t \,f_t\,, \nabla_t H_t \big\rangle_t \Omega \;\; = \;\;
  \int_X \Delta_t \,H_t\, \Omega \;\;=\;\; \int_X \left\langle \mathbbm{I}_{_{T_X}},
  \partial^{g_t}_{_{T_{X, J_t}}} \nabla_t H_t \right\rangle_t \Omega \;.
\end{eqnarray*}

\section{Appendix}

\subsection{The first variation of Perelman's $\mathcal{W}$ functional along
the K\"ahler-Ricci flow}\label{Per-first-varKRF}

Analogues of the following evolution formulas were obtained by Perelman
\cite{Per}
in the Ricci flow case. 
For notation convenience we denote by $(g_t)_{t
\geqslant 0}$ the K\"ahler-Ricci flow and with $\omega_t$ the corresponding
symplectic forms.

\begin{theorem}
{\tmstrong{$(\tmop{Perelman})$}} 
  Let $X$ be a Fano manifold and let $f$ be a solution of the conjugate heat
  equation
  \begin{equation}
    \label{BackHet-f} 2 \;\dot{f} \;\;= \;\;-\;\,  \Delta \,f \;\; +\;\;
     | \nabla f|^2 \;\;+\;\; 2\, n \;\;-\;\; \tmop{Scal}\;,
  \end{equation}
  along the K\"ahler-Ricci flow $(g_t)_{t \geqslant 0}$, over a time interval
  $\left[ 0, T \right]$. Then the function
  \begin{eqnarray*}
    2 \;H \;\;\assign\;\; 2 \;\Delta f \; - \; | \nabla f|^2 \;\;+\;\;
    \tmop{Scal} \;\;+\;\; 2 \;f \;\;-\;\; 2\; n\;,
  \end{eqnarray*}
  satisfies the evolution equation
  \begin{equation}
    \label{ev-H} 2 \;\dot{H} \;\;=\;\; -\;\; \Delta\, H + 2 \;\nabla H \cdot \nabla f
\;\; + \;\; 
| \tmop{Ric} \;\;+\;\;
    i \partial \overline{\partial} f \;\; -\;\;
    \omega_t |^2 \;\;+ \;\; | \nabla^{1, 0} \partial
    f|^2 \;,
  \end{equation}
  over the time interval $\left[ 0, T \right]$. Moreover on this interval hold
  the variation formula
\begin{equation}
    \label{W-frst}  \frac{d}{dt} \hspace{0.25em} \mathcal{W} (g_t, f_t)
    = \int_X \Big[ | \tmop{Ric} \;\;+ \;\; i \partial
    \overline{\partial} f \;\; -\;\; \omega_t |^2 \;\; +
    \;\; | \nabla^{1, 0} \partial f|^2 \Big] e^{- f} d
    V_{g_t}  \hspace{0.25em} .
\end{equation}
\end{theorem}

\begin{proof}
  We remind first that for any function 
$f \in C^{\infty} (X \times \R_{\geqslant 0} \hspace{0.25em}, \R)$ 
hold the evolution formulas along the K\"ahler-Ricci
  flow
  \begin{eqnarray}
    \frac{\partial}{\partial t} \; | \nabla f|^2 & = & -
    \;\;| \nabla f|^2 \hspace{0.75em} + \hspace{0.75em} \tmop{Ric}
    (\nabla f, J \nabla f) \hspace{0.75em} + \hspace{0.75em} 2 \;\nabla \dot{f}
    \cdot \nabla f \hspace{0.25em},  \label{time-grad}\\
    &  &  \nonumber\\
    \Delta\, | \nabla f|^2 & = & 2 \;\nabla \Delta f \cdot \nabla f
    \hspace{0.75em} + \hspace{0.75em} 2\; | \nabla^{1, 0} \partial f|^2
    \hspace{0.75em} + \hspace{0.75em} 2 \;| \partial \overline{\partial} f|^2
    \nonumber\\
    &  &  \nonumber\\
    & + & 2 \tmop{Ric} (\nabla f, J \nabla f) \hspace{0.25em}, 
    \label{LapGradEvol}\\
    &  &  \nonumber\\
    \frac{\partial}{\partial t} \hspace{0.25em} \Delta f & = & -
    \;\;\Delta \,f \;\; +\;\; \left\langle \tmop{Ric}, i
    \partial \overline{\partial} f \right\rangle \;\;+\;\; \Delta\,
    \dot{f} \hspace{0.25em} .  \label{time-Lap}
  \end{eqnarray}
  We remind also that the scalar curvature evolves by the formula
  \begin{equation}
    \label{ev-scal} 2 \hspace{0.25em}  \frac{\partial}{\partial t}
    \hspace{0.25em} \tmop{Scal} \hspace{0.75em} = \hspace{0.75em} \Delta
    \tmop{Scal} \hspace{0.75em}+\hspace{0.75em} 2\;| \tmop{Ric} |^2 \hspace{0.75em} -
    \hspace{0.75em} 2 \;\tmop{Scal} \hspace{0.25em} .
  \end{equation}
Furthermore we observe the identity
  \begin{equation}
    \label{Harnk=Heat} 2 H \hspace{0.75em} = \hspace{0.75em} \Box
    \hspace{0.25em} f \hspace{0.75em} + \hspace{0.75em} 2 \;f \hspace{0.25em} .
  \end{equation}
  By using the evolution equation (\ref{time-Lap}) we get the equality
  \begin{eqnarray*}
    \Box \hspace{0.25em} \Delta \,f & = & \Delta^2 f \hspace{0.75em} +
    \hspace{0.75em} 2 \;\Delta \,f \hspace{0.75em} - \hspace{0.75em} 2\;
    \left\langle \tmop{Ric}, i \partial \overline{\partial} f \right\rangle
    \hspace{0.75em} - \hspace{0.75em} 2 \;\Delta \,\dot{f}\\
    &  & \\
    & = & \Delta \hspace{0.25em} \Box \hspace{0.25em} f \hspace{0.75em} +
    \hspace{0.75em} 2 \;\Delta \,f \hspace{0.75em} - \hspace{0.75em} 2
    \left\langle \tmop{Ric}, i \partial \overline{\partial} f \right\rangle\\
    &  & \\
    & = & 2 \;\Delta \,H \hspace{0.75em} - \hspace{0.75em} 2 \left\langle
    \tmop{Ric}, i \partial \overline{\partial} f \right\rangle
    \hspace{0.25em},
  \end{eqnarray*}
  thanks to the identity (\ref{Harnk=Heat}). Moreover if we combine the
  evolution equations (\ref{time-grad}) and (\ref{LapGradEvol}) we obtain
  \begin{eqnarray}
    \Box\,| \nabla f|^2 & = & 2 \;\nabla \hspace{0.25em} \Box \hspace{0.25em} f
    \cdot \nabla f \hspace{0.75em} + \hspace{0.75em} 2 \;| \nabla^{1, 0}
    \partial f|^2 \hspace{0.75em} + \hspace{0.75em} 2 \;| \partial
    \overline{\partial} f|^2 \hspace{0.75em} + \hspace{0.75em} 2 \;| \nabla f|^2
    \nonumber\\
    &  &  \nonumber\\
    & = & 4 \;\nabla H \cdot \nabla f \hspace{0.75em} + \hspace{0.75em} 2 \;|
    \nabla^{1, 0} \partial f|^2 \hspace{0.75em} + \hspace{0.75em} 2\;| \partial
    \overline{\partial} f|^2 \hspace{0.75em} - \hspace{0.75em} 2 \;| \nabla f|^2
    \hspace{0.25em}, \nonumber
  \end{eqnarray}
  thanks to the identity (\ref{Harnk=Heat}). We infer the expressions
  \begin{eqnarray*}
    2\hspace{0.25em} \Box\, H & = & 2 \hspace{0.25em} \Box \hspace{0.25em}
    \Delta \,f \hspace{0.75em} - \hspace{0.75em} \Box \hspace{0.25em} | \nabla
    f|^2 \hspace{0.75em} + \hspace{0.75em} \Box \hspace{0.25em} \tmop{Scal}
    \hspace{0.75em} + \hspace{0.75em} 2 \hspace{0.25em} \Box \hspace{0.25em}
    f\\
    &  & \\
    & = & 4 \;\Delta \,H \hspace{0.75em} - \hspace{0.75em} 4 \left\langle
    \tmop{Ric}, i \partial \overline{\partial} f \right\rangle\\
    &  & \\
    & - & 4 \;\nabla \,H \cdot \nabla f \hspace{0.75em} - \hspace{0.75em} 2 \;|
    \nabla^{1, 0} \partial f|^2 \hspace{0.75em} - \hspace{0.75em} 2\; | \partial
    \overline{\partial} f|^2 \hspace{0.75em} + \hspace{0.75em} 2 \;| \nabla
    f|^2\\
    &  & \\
    & - & 2 \;| \tmop{Ric} |^2 \hspace{0.75em} + \hspace{0.75em} 2
    \tmop{Scal} \hspace{0.75em} + \hspace{0.75em} 4 \;H \hspace{0.75em} -
    \hspace{0.75em} 4 \;f
\end{eqnarray*}
\begin{eqnarray*}
    & = & 4 \;\Delta \,H \hspace{0.75em} + \hspace{0.75em} 4 \;\Delta \,f
    \hspace{0.75em} + \hspace{0.75em} 4\; \tmop{Scal} \hspace{0.75em} -
    \hspace{0.75em} 4 \;\left\langle \tmop{Ric}, i \partial \overline{\partial}
    f \right\rangle \hspace{0.75em} - \hspace{0.75em} 4 \;\nabla H \cdot \nabla
    f\\
    &  & \\
    & - & 2 \;| \nabla^{1, 0} \partial f|^2 \hspace{0.75em} - \hspace{0.75em} 2\;
    | \partial \overline{\partial} f|^2 \hspace{0.75em} - \hspace{0.75em} 2 \;|
    \tmop{Ric} |^2 \hspace{0.75em} - \hspace{0.75em} 4 \;n\\
    &  & \\
    & = & 4 \;\Delta \,H \hspace{0.75em} - \hspace{0.75em} 4 \;\nabla H \cdot
    \nabla f 
\\
\\
&-&
2 \,(B \hspace{0.75em} -
    \hspace{0.75em} 2 \;\Delta \,f \hspace{0.75em} + \hspace{0.75em} 2\, n
    \hspace{0.75em} - \hspace{0.75em} 2 \;\tmop{Scal}) \hspace{0.25em} .
  \end{eqnarray*}
  where
  \[ B \hspace{0.75em} \;\;: =\;\; | \nabla^{1, 0} \partial f|^2 \;\;+\;\; | \partial
     \overline{\partial} f|^2 \hspace{0.75em} + \hspace{0.75em} 2 \left\langle
     \tmop{Ric}, i \partial \overline{\partial} f \right\rangle
     \hspace{0.75em} + \hspace{0.75em}| \tmop{Ric} |^2 \hspace{0.25em} . \]
  Arranging the terms by means of the trivial identity $\tmop{Tr}_{\omega_t}
  \alpha = \left\langle \omega_t, \alpha \right\rangle$, with $\alpha$ a real
  $(1, 1)$-form, we obtain the evolution equation
\begin{eqnarray*}
    \label{Het-H} 2 \hspace{0.25em} \Box \hspace{0.25em} H 
& =& 4 \hspace{0.25em} \Delta \,H \hspace{0.75em} -
    \hspace{0.75em} 4 \hspace{0.25em} \nabla H \cdot \nabla f 
\\
\\
&-&
2 \hspace{0.25em} | \tmop{Ric} \hspace{0.75em} +
    \hspace{0.75em} i \partial \overline{\partial} f \hspace{0.75em} -\hspace{0.75em}
    \omega_t |^2 \hspace{0.75em} - \hspace{0.75em} 2 \hspace{0.25em} |
    \nabla^{1, 0} \partial f|^2 \hspace{0.25em},
\end{eqnarray*}
  which implies the evolution formula (\ref{ev-H}). We remind now that the
  evolution equation (\ref{BackHet-f}) rewrites as $\Box^{\ast} e^{- f} = 0$.
  Thus time deriving the identity
  \[ \mathcal{W} (g_t, f_t) \hspace{0.75em} = \int_X 2\, H \hspace{0.25em} e^{-
     f} d V_g  \hspace{0.25em}, \]
  we infer
  \begin{eqnarray*}
    \frac{d}{dt} \hspace{0.25em}  \mathcal{W} (g_t, f_t) & = & -
    \hspace{0.75em} \int_X \Box \hspace{0.25em} H \hspace{0.25em} e^{- f} d
    V_{g}  \hspace{0.25em},
  \end{eqnarray*}
  which implies Perelman's variation formula (\ref{W-frst}).
\end{proof}

\subsection{Local expression of the complex anti-linear part of the Hessian }\label{loc-expr-antH}

Let $(X, J, \omega)$ be a K\"ahler manifold and $u \in C^2 (X, \mathbbm{R})$.
Let $(z_1, \ldots, z_n)$ be $J$-holomorphic coordinates and consider the local
expression
\[ \overline{\partial}_{_{T_{X, J}}} \nabla_g \,u \;\; =\;\;
    A_{k, \bar{l}} \;\, \bar{\zeta}_l^{\ast} \otimes
   \zeta_k \hspace{0.75em} + \hspace{0.75em} \overline{A_{k, \bar{l}}}
   \;\, \zeta_l^{\ast} \otimes \bar{\zeta}_k  \hspace{0.25em}, \]
where $\zeta_k : = \frac{\partial}{\partial z_k}$. 
We want to find the
expression of the coefficients $A_{k, \bar{l}}$ with respect to $u$. For this
purpose we consider the identities 
$$
\nabla_g \,u \;\;=\;\; \nabla^{1, 0}_{g, J} \,u \;\;+\;\;
\nabla^{0, 1}_{g, J} \,u
$$ 
and
\[ \nabla^{1, 0}_{g, J} \,u \;\neg\; \omega \hspace{0.75em} = \hspace{0.75em} i\,
   \overline{\partial}_{_J} u \hspace{0.25em} . \]
If we write locally $\nabla^{1, 0}_{g, J} \,u \;=\; \xi_k \;\zeta_k$ then the last
identity writes locally as
\[ \frac{i}{2} \hspace{0.25em} \omega_{l, \bar{k}}  \hspace{0.25em} \xi_l
   \; \bar{\zeta}_k^{\ast} \hspace{0.75em} = \hspace{0.75em} i
   \hspace{0.25em} ( \bar{\zeta}_k \hspace{0.25em} . \hspace{0.25em} u)
   \; \bar{\zeta}_k^{\ast} \hspace{0.25em} . \]
We infer the expression $\xi_l = 2 \,\omega^{k, \bar{l}} \hspace{0.25em}
\bar{\zeta}_k \hspace{0.25em} . \hspace{0.25em} u$. Moreover by the definition
of the operator $\overline{\partial}_{_{T_{X, J}}}$ hold the identities
\[  A_{k, \bar{l}} \; \zeta_k \hspace{0.75em} =
   \hspace{0.75em}  \left[ \left( \overline{\partial}_{_{T_{X, J}}} \nabla_g \,u
   \right)  \bar{\zeta}_l \right]^{1, 0}_{_{J_t}} \hspace{0.75em} =
   \hspace{0.75em}  \left[ \bar{\zeta}_l \hspace{0.25em}, \hspace{0.25em}
   \nabla^{1, 0}_{g, J} \,u \right]^{1, 0}_{_{J_t}} \hspace{0.75em} =
   \hspace{0.75em} ( \bar{\zeta}_l \; .\; \xi_k) \hspace{0.25em}
   \zeta_k \hspace{0.25em} . \]
We infer the expressions
\begin{eqnarray*}
  A_{k, \bar{l}} & = & \bar{\zeta}_l \hspace{0.25em} . \hspace{0.25em} \xi_k
  \hspace{0.75em} = \hspace{0.75em} 2 \bar{\zeta}_l \hspace{0.25em} . \left(
  \omega^{r, \bar{k}}  \hspace{0.25em} \bar{\zeta}_r \hspace{0.25em} .
  \hspace{0.25em} f \right)\\
  &  & \\
  & = & 2 \omega^{p, \bar{k}}  \left[ \bar{\zeta}_l \hspace{0.25em} .
  \hspace{0.25em} \bar{\zeta}_p \hspace{0.25em} . \hspace{0.25em} f_t
  \hspace{0.75em} - \hspace{0.75em} \left( \bar{\zeta}_l \hspace{0.25em} .
  \hspace{0.25em} \omega_{j, \bar{p}} \right) \hspace{0.25em} \omega^{r,
  \bar{j}} \hspace{0.25em} \bar{\zeta}_r \hspace{0.25em} . \hspace{0.25em} f_t
  \right] \hspace{0.25em} .
\end{eqnarray*}
\subsection{Comparison of norms on $T_{_{X, J}}$-valued forms}\label{comp-CX-norms}

Let $(X, J, g)$ be a hermitian manifold. Let $\omega : = g J$ and let
$h^{\ast}$ be the corresponding hermitian metric over the complex vector
bundle $T^{\ast}_{_{X, J}}$. With respect to a local complex frame
$(\zeta_k)_k \subset T_{_{X, J}}^{1, 0}$ we have the expression
\[ h^{\ast} \;\;=\;\; 4 \, \sum_{k, l} \hspace{0.25em} \omega^{l \bar{k}}
   \hspace{0.25em} \zeta_k \otimes \bar{\zeta}_l \;. \]
We remind that if $(V, J)$ is a complex vector space equipped with a hermitian
metric $h$ then the corresponding hermitian metric $h_{_{\C}}$ over the
complexified vector space $(V \otimes_{_{\R}} \C, i)$ is defined by the
formula
\[ 2 \,h_{_{\C}} (v, w) \;\;: =\;\; h (v, \overline{w}) \;\,+\;\, \overline{h ( \overline{v},
   w)}, \hspace{1em} v, w \in V \otimes_{_{\R}} \C\;, \]
where we still note by $h$ the $\C$-linear extension of $h$. Thus $h_{_{\C}}$
coincides with the sesquilinear extension over $V \otimes_{_{\R}} \C$ of the
Riemannian metric associated to $h$. We infer by the expression (31) in \cite{Pal} of the
Riemannian metric on the exterior products that the induced hermitian product on the
vector bundle $\Lambda^{p, q}_J T^{\ast}_X$ is given by the formula
\begin{eqnarray*}
  &  & \Big\langle \Lambda_{j = 1}^p \alpha_{1, j} \wedge \Lambda_{j = 1}^q
  \beta_{1, j} \hspace{0.25em}, \hspace{0.25em} \Lambda_{j = 1}^p \alpha_{2,
  j} \wedge \Lambda_{j = 1}^q \beta_{2, j} \Big\rangle\\
  &  & \\
  & = & (p + q) ! \det \left( 2^{- 1} h^{\ast} (\alpha_{1, j},
  \bar{\alpha}_{2, l}) \right) \hspace{0.25em}  \overline{\det \left( 2^{- 1}
  h^{\ast} ( \bar{\beta}_{1, j}, \beta_{2, l}) \right)} \;.
\end{eqnarray*}
Consider now an element 
$$
A \;\in\; T^{\ast}_{_{X, - J}} \otimes_{_{\mathbbm{C}}} T_{_{X, J}}
\;\cong\; \Lambda^{0, 1}_J T^{\ast}_X \otimes_{_{\mathbbm{C}}} T_{_{X, J}}\;,
$$ 
and
let $(e_k)_k \subset T_{_{X, J}}, e_k \assign \zeta_k \,+\, \bar{\zeta}_k$ be the
$J$-complex basis associated to $(\zeta_k)_k$. Then hold the local expression
\begin{eqnarray*}
  A \;\; = \;\; A_{k, \bar{l}}\;  \bar{\zeta}^{\ast}_l \otimes_{_J} e_k \;\;=\;\; A_{k,
  \bar{l}}  \;\bar{\zeta}^{\ast}_l \otimes \zeta_k \;\,+\;\, \tmop{Conjugate} \;.
\end{eqnarray*}
Assume from now on that the frame $(e_k)_k$ is $h$-orthonormal. On one side if
one think of $A$ as an element in $\tmop{End}_{_{\mathbbm{R}}} \left( T_{_X}
\right)$ then
\begin{eqnarray*}
  |A|^2_g \;\; = \;\; \tmop{Tr}_{_{\mathbbm{R}}} \left( A\, A_g^T \right) \;\;=\;\; 2\, |A_{k,
  \bar{l}} |^2\;,
\end{eqnarray*}
since
\begin{eqnarray*}
  A_g^T \;\; = \;\; A_{l, \bar{k}} \; \bar{\zeta}^{\ast}_l \otimes \zeta_k \;\,+\;\,
  \tmop{Conjugate} \;.
\end{eqnarray*}
On the other side
\begin{eqnarray*}
  |A|^2_{\omega} \;\;\equiv\;\; |A|_{\Lambda_J^{0, 1} T^{\ast}_X
  \otimes_{_{\mathbbm{C}}} T_{X, J}, \omega}^2 
& = & \left\langle A_{k,
  \bar{l}}  \;\bar{\zeta}^{\ast}_l\,, A_{k, \bar{p}}  \;\bar{\zeta}^{\ast}_p
  \right\rangle \\
  &  & \\
  & = & \frac{1}{2}  \,\overline{h^{\ast} \left( \overline{A}_{k, \bar{l}}\;
  \zeta^{\ast}_l\,, A_{k, \bar{p}}  \;\bar{\zeta}^{\ast}_p \right)}\\
  &  & \\
  & = & \frac{1}{2} \,|A_{k, \bar{l}} |^2 \cdot 4\;,
\end{eqnarray*}
Thus $|A|^2_g = |A|^2_{\omega} .$ The same identity hold true for any 
$$
A \;\in\;
T^{\ast}_{_{X, J}} \otimes_{_{\mathbbm{C}}} T_{_{X, J}} \;\cong\; \Lambda^{1, 0}_J
T^{\ast}_X \otimes_{_{\mathbbm{C}}} T_{_{X, J}}\;.
$$ 
In higher degrees there is a multiplicative factor involved. We consider for example $A \in
\Lambda^{1, 1}_J T^{\ast}_X \otimes_{_{\mathbbm{C}}} T_{_{X, J}}$ and its
local expression
\begin{eqnarray*}
  A \;\; = \;\; i\, A_{p, k, \bar{l}}\,  \left( \zeta_p^{\ast} \wedge
  \bar{\zeta}_l^{\ast} \right) \otimes \zeta_k \;\,+\;\, \tmop{Conjugate}\; .
\end{eqnarray*}
Then
\begin{eqnarray*}
  |A|_{\Lambda_J^{1, 1} T^{\ast}_X \otimes_{_{\mathbbm{C}}} T_{X, J},
  \omega}^2 & = & \left\langle i \,A_{p, k, \bar{l}} \;\zeta_p^{\ast} \wedge
  \bar{\zeta}_l^{\ast}\,, i\, A_{r, k, \bar{h}} \;\zeta_r^{\ast} \wedge
  \bar{\zeta}_h^{\ast} \right\rangle\\
  &  & \\
  & = & A_{p, k, \bar{l}}  \;\overline{A}_{r, k, \bar{h}}  \left\langle
  \zeta_p^{\ast} \wedge \bar{\zeta}_l^{\ast}\,, \zeta_r^{\ast} \wedge
  \bar{\zeta}_h^{\ast} \right\rangle\\
  &  & \\
  & = & \frac{1}{2} \,|A_{p, k, \bar{l}} |^2 \,h^{\ast} \left( \zeta_p^{\ast}\,,
  \bar{\zeta}_p^{\ast} \right)  \,\overline{h^{\ast} \left( \zeta_l^{\ast}\,,
  \bar{\zeta}_l^{\ast} \right)}\\
  &  & \\
  & = & 8\, |A_{p, k, \bar{l}} |^2 \;.
\end{eqnarray*}
On the other side if we think of $A$ as an element of $\Lambda^2 T_X
\otimes_{_{\mathbbm{R}}} T_X$ then \
\begin{eqnarray*}
  |A|_{\Lambda^2 T_X \otimes_{_{\mathbbm{R}}} T_X, g}^2 
& = &
  \tmop{Tr}_{_{\mathbbm{R}}} \left[ \left( e_r \;\neg\; A \right) \left( e_r \;\neg\;
  A \right)_g^T \right] 
\\
\\
&+& 
\tmop{Tr}_{_{\mathbbm{R}}} \left[ \left( J e_r \;\neg\;
  A \right) \left( J e_r \;\neg\; A \right)_g^T \right]\;,
\end{eqnarray*}
and
\begin{eqnarray*}
  e_r \;\neg\; A & = & i\, A_{r, k, \bar{l}}  \;\bar{\zeta}_l^{\ast} \otimes \zeta_k 
\;\,-\;\,
  i \,A_{l, k, \bar{r}} \;\zeta_l^{\ast} \otimes \zeta_k \;\,+\;\, \tmop{Conjugate}\;,\\
  &  & \\
  J e_r \;\neg\; A & = & - \;\,A_{r, k, \bar{l}}  \;\bar{\zeta}_l^{\ast} \otimes \zeta_k
  \;\,-\;\, A_{l, k, \bar{r}} \;\zeta_l^{\ast} \otimes \zeta_k \;\,+\;\, \tmop{Conjugate}\;,\\
  &  & \\
  \left( e_r \;\neg\; A \right)_g^T & = & i\, A_{r, l, \bar{k}} \;
  \bar{\zeta}_l^{\ast} \otimes \zeta_k \;\,+\;\, i\, \overline{A}_{k, l, \bar{r}}\;
  \zeta_l^{\ast} \otimes \zeta_k \;\,+\;\, \tmop{Conjugate}\;,\\
  &  & \\
  \left( J e_r \;\neg\; A \right)_g^T & = & - \;\,A_{r, l, \bar{k}} \;
  \bar{\zeta}_l^{\ast} \otimes \zeta_k \;\,-\;\, \overline{A}_{k, l, \bar{r}}\;
  \zeta_l^{\ast} \otimes \zeta_k \;\,+\;\, \tmop{Conjugate}\; .
\end{eqnarray*}
(By conjugate we mean the complex conjugate of all therms preceding this
  word.) Thus hold the equality
\begin{eqnarray*}
  \tmop{Tr}_{_{\mathbbm{R}}} \left[ \left( e_r \;\neg\; A \right) \left( e_r \neg
  A \right)_g^T \right] \;\; = \;\; 2\, |A_{l, k, \bar{r}} |^2 \;\;=\;\;
  \tmop{Tr}_{_{\mathbbm{R}}} \left[ \left( J e_r \;\neg\; A \right) \left( J e_r
  \;\neg\; A \right)_g^T \right]\; .
\end{eqnarray*}
We infer the identity
\begin{eqnarray*}
  2\,|A|_{\Lambda^2 T_X \otimes_{_{\mathbbm{R}}} T_X, g}^2 \;\; = \;\;
  |A|_{\Lambda_J^{1, 1} T^{\ast}_X \otimes_{_{\mathbbm{C}}} T_{X, J},
  \omega}^2 \;.
\end{eqnarray*}

\vspace{1cm}
\noindent
Nefton Pali
\\
Universit\'{e} Paris Sud, D\'epartement de Math\'ematiques 
\\
B\^{a}timent 425 F91405 Orsay, France
\\
E-mail: \textit{nefton.pali@math.u-psud.fr}
\end{document}